%% file: main.tex
\documentclass[11pt,colorlinks=true,allcolors=OliveGreen]{article}
\usepackage[english]{babel}
\usepackage{setspace}
\usepackage[letterpaper,top=2cm,bottom=2cm,left=2cm,right=2cm,marginparwidth=1.75cm]{geometry}
\usepackage{amssymb}
\usepackage{amsmath}
\usepackage{amsfonts}
\usepackage{amsthm}
\usepackage{mathtools}
\usepackage{mathrsfs}
\usepackage{bm}
\usepackage{bbm}          
\usepackage{dsfont}
\usepackage{cancel}
\usepackage{physics}      
\usepackage{stmaryrd}     
\usepackage{tensor}
\usepackage{systeme}
\usepackage{upgreek}
\usepackage{nicefrac}
\usepackage{extarrows}
\usepackage{thmtools}
\usepackage{stackengine}
\stackMath
\usepackage{esint}
\usepackage{calrsfs}

\usepackage{array}
\newcolumntype{C}{>{\centering\arraybackslash}p{2.5cm}}
\usepackage{tabularx}
\usepackage{longtable}
\usepackage{colortbl}
\usepackage{booktabs}
\usepackage{multirow}
\usepackage{nicematrix}
\usepackage{enumitem}
\usepackage{siunitx}
\usepackage{float}
\usepackage{placeins}
\usepackage{graphicx}
\graphicspath{{Images/}}
\usepackage[export]{adjustbox}
\usepackage{overpic}
\usepackage{eso-pic}
\usepackage{transparent}
\usepackage{subcaption}
\usepackage[labelfont=bf,labelsep=space]{caption}

\usepackage{tikz}
\usepackage{pgfplots}
\pgfplotsset{compat=newest}
\usepgfplotslibrary{units}
\usetikzlibrary{
  pgfplots.groupplots,
  external,
  patterns, tikzmark,
  matrix, decorations.pathreplacing, calc,
  arrows.meta, positioning, shapes.geometric
}
\usepackage{hf-tikz}
\pgfkeys{/pgfplots/legend entry/.code=\addlegendentry{#1}}

\usepackage{algorithm}
\usepackage{algpseudocode}
\usepackage{csvsimple}
\usepackage{comment}
\usepackage{pdfpages}
\usepackage{lipsum}
\usepackage{xargs}
\usepackage{abstract}

\usepackage{authblk}

\usepackage{orcidlink}

\usepackage[dvipsnames,x11names]{xcolor}
\usepackage{hyperref}

\usepackage[overload]{empheq}

\newcommand{\faces}{\mathcal{F}_h}

\newcommand{\mc}[1]{\textcolor{purple}{[\textbf{Mattia:} #1]}}
\newcommand{\pa}[1]{\textcolor{magenta}{[\textbf{Paola:} #1]}}

\newcommand{\eremk}{\hbox{}\hfill\rule{0.8ex}{0.8ex}}

\newcommand{\jump}[1]{\llbracket #1 \rrbracket}
\newcommand{\avg}[1]{\{\!\!\{ #1 \}\!\!\}}
\newcommand{\dgnorm}[1]{\|#1\|_\mathrm{DG}}
\newcommand{\tdgnorm}[1]{|\;\!\!|\;\!\!|#1|\;\!\!|\;\!\!|_\mathrm{DG}}
\newcommand{\enorm}[1]{\|#1\|_\varepsilon}

\newcommand{\ds}{\mathrm{d}s}
\newcommand{\Wh}{W_h^\mathrm{DG}}
\newcommand{\bea}{\begin{eqnarray}}
\newcommand{\eea}{\end{eqnarray}}

\newtheorem{theorem}{Theorem}[section]

\theoremstyle{definition}

\newtheorem{proposition}[theorem]{Proposition}
\newtheorem{assumption}{Assumption}[section]
\theoremstyle{remark}
\newtheorem{remark}[theorem]{Remark}

\tikzset{font={\fontsize{15pt}{12}\selectfont}}

\title{A stability-preserving polytopal discontinuous Galerkin method for the Fisher-Kolmogorov model with applications to neurodegenerative disease modelling\footnote{\textbf{Fundings}: This work is partially funded by the European Union (ERC SyG, NEMESIS, project number 101115663). Views and opinions expressed are, however, those of the authors only and do not necessarily reflect those of the European Union or the European Research Council Executive Agency. Neither the European Union nor the granting authority can be held responsible for them. The present research is part of the activities of the Dipartimento di Eccellenza 2023-2027 grant, funded by MUR. PFA, FB, MC, NDM, and FR are members of INdAM-GNCS. }}

\author[1]{Paola F. Antonietti \orcidlink{0000-0002-2138-3878}}
\author[1]{Francesca Bonizzoni \orcidlink{0000-0002-6222-3352}}
\author[1]{Mattia Corti \orcidlink{0000-0002-7014-972X}}
\author[1]{Nicola De March\orcidlink{0009-0005-8288-7964}}
\author[1]{Salvatore Di Noto}
\author[1]{Francesco Regazzoni \orcidlink{0000-0002-4207-1400}}

\affil[1]{MOX - Modeling and Scientific Computing Laboratory, Dipartimento di Matematica, Politecnico di Milano, Piazza Leonardo da Vinci 32, Milan, 20133, Italy}

\AtBeginDocument{\RenewCommandCopy\qty\SI}
\begin{document}
\date{}
\maketitle

\begin{abstract}
The Fisher--Kolmogorov equation models the spatio-temporal evolution of interacting biological species and is extensively employed in fields such as ecology, population dynamics, and the modelling of neurodegenerative diseases. Under suitable assumptions on the data, the solution $c$ is non-negative, a key feature because $c$ typically denotes a population density or a relative concentration. However, standard discretisation methods often fail to preserve this property, resulting in non-physical oscillations and unstable numerical approximations. In this work, we propose and analyse a numerical method to stabilise the dynamics of the Fisher--Kolmogorov model around the unstable equilibrium $c=0$. The proposed approach combines a discontinuous Galerkin spatial discretisation on general polygonal and polyhedral meshes with the Crank--Nicolson time integration scheme. The main idea is to suitably modify the formulation at the continuous level so that, on the one hand, it is strongly consistent with the original model, and, on the other hand,  it ensures stability when moving to the discrete setting.  We prove well-posedness of the semi-discrete formulation, derive stability bounds and prove optimal \textit{a priori} error estimates in a suitable energy norm. The theoretical results are demonstrated through a comprehensive set of numerical examples. Moreover, we consider an application arising in computational neuroscience by simulating the propagation of $\alpha$-synuclein, a key pathogenic protein implicated in Parkinson's disease and other neurodegenerative diseases, demonstrating that the proposed scheme is stable, high-order accurate, and robust in a biologically relevant computational setting.
\end{abstract}

\section{Introduction}
\label{sec:introduction}
The Fisher--Kolmogorov (FK) equation is a time-dependent, non-linear reaction--diffusion model combining diffusion with a logistic reaction term \cite{salsa_partial_2022}. Such a model is widely employed in many application fields, such as ecology, population dynamics, tumour growth, and chemical and environmental engineering.  More recently, it has been widely employed to describe the propagation of prion-like proteins in neurodegenerative diseases, including Alzheimer's and Parkinson's diseases \cite{walker_neurodegenerative_2015,chen_amyloid_2017,jouanne_tau_2017,calabresi_alpha-synuclein_2023}. More detailed descriptions based on the Smoluchowski equation have been proposed to model protein aggregation kinetics \cite{franchi_from_2016,bertsch_alzheimer_2017}, however, their computational complexity has motivated the development of reduced reaction--diffusion models, such as the heterodimer model and the FK equation \cite{weickenmeier_physics_2019,fornari_prion-like_2019}.

The widespread use of the FK equation in different applications has stimulated the development of a variety of numerical methods for its approximation. Several approaches have been investigated, including finite element methods \cite{weickenmeier_physics_2019, engwer_estimating_2021, roessler_numerical_1997}, finite difference schemes \cite{macias-diaz_explicit_2012}, boundary element methods \cite{gortsas_local_2022}, and discontinuous Galerkin  methods \cite{corti_discontinuous_2023, Naldi2013DiscontinuousGA}. 
A key important feature of the FK model is that, under suitable assumptions on the data, the  solution $c$ is non-negative and bounded, a key feature because $c$ typically denotes a population density or a relative concentration. In particular, the preservation of non-negativity plays a central role in applications involving biological quantities, such as the diffusion of misfolded proteins in the brain. At the discrete level, however, standard discretisation methods often fail to preserve this property, resulting in non-physical oscillations and unstable numerical approximations. Consequently, recent works have focused on the design of structure-preserving formulations capable of enforcing such a property at the discrete level. For example, a first class of methods enforces positivity/boundedness in an algebraic, a-posteriori fashion via scaling limiters, flux-correction, or Lagrange-multiplier approaches; see, e.g., \cite{sun_discontinuous_2019, amiri_nodally_2024, cheng_lagrange_2022, huang_bound_2023}. Instead, a second class of methods relies on a non-linear change of variables, expressing the unknown in terms of an entropy (or exponential) variable, so that positivity/boundedness constraints are automatically satisfied, see, e.g.,   \cite{bonizzoni_structure_2020,corti_structure_2024,antonietti_structure-preserving_2026, gomez_structure_2026}. While effective, methods based on a change of variables introduce additional non-linearities in the differential operators and typically lead to more ill-conditioning of the discrete systems, complicating the design of efficient non-linear solvers.

\par
In this work, we propose and analyse a numerical method based on a stabilised version of the FK model (SFK). The main idea concerns the non-linear reaction term, for which an absolute-value stabilization is introduced. On the one hand, the modified formulation is strongly consistent with the original model at the continuous level, thus it preserves the original structure of the equation. On the other hand, it automatically ensures stability, when moving to the discrete setting, avoiding unphysical oscillations. To approximate the solution of the proposed SFK model, we propose a high-order discontinuous Galerkin method on general polygonal and polyhedral meshes (PolyDG) for the spatial discretisation, coupled with a second-order implicit time integration scheme such as Crank--Nicolson. This choice ensures high-order accuracy and enables the treatment of complex geometries, providing significant flexibility in mesh generation, particularly for the application of interest.
\par
We present a complete analysis of the semi-discrete formulation, establishing improved stability and convergence estimates compared to the original formulation in \cite{corti_discontinuous_2023}, in both two and three dimensions. Moreover, the structure of the method naturally extends the analysis to general polytopal meshes \cite{cangiani_hp-version_2014,cangiani_hp-version_2017}, which is especially relevant for brain applications where mesh agglomeration strategies are particularly crucial to address the inherit geometric complexity \cite{antonietti_polytopal_2026}.

\par
The theoretical results are validated through a comprehensive set of numerical experiments. We also demonstrate the practical performance of the proposed formulation by considering an application in computational neuroscience, namely the simulation of protein propagation, which is a hallmark of neurodegenerative diseases. In particular, we consider the propagation of $\alpha$-synuclein, a protein implicated in Parkinson's disease and other neurodegenerative disorders, within a brain section reconstructed from medical images. The resulting diffusion patterns of misfolded proteins are consistent with observations reported in the medical literature \cite{braak_staging_2003,goedert_alzheimer_2015}. A comparison between the SFK and the standard FK models further shows that the former provides reliable results even with lower polynomial degrees, making it a robust and computationally efficient approach.

\par
\bigskip
This paper is organised as follows. In Section~\ref{sec:FKPP}, we recall the FK model and its functional setting and we introduce the SFK model. In Section~\ref{sec:PolyDG}, we present some preliminary technical backgrounds and the PolyDG space discretisation of the problem. In Section~\ref{sec:stability}, we prove the stability bounds and in  Section~\ref{sec:errorAnalysis}, we derive an \textit{a-priori} error estimate of the semi-discretised problem. In Section~\ref{sec:discreteMFK}, we introduce the Crank--Nicolson time discretisation scheme for the SFK model. In Section~\ref{sec:verification} and Section~\ref{sec:validation}, we report the numerical results, starting from some convergence tests in the case of a given manufactured solution and the simulations of the diffusion of $\alpha$-synuclein protein in Parkinson's disease in a two-dimensional brain section reconstructed from medical images, respectively. Finally, in Section~\ref{sec:end}, we draw some conclusions and discuss possible future developments.

\section{The Fisher-Kolmogorov model and its stabilised version}
\label{sec:FKPP}
In this section, we present the FK model, used to describe the reaction and diffusion of interacting species \cite{weickenmeier_physics_2019}. We consider the domain $\Omega \subset \mathbb{R}^d$ with $d= 2, 3$ and the time interval $(0,T)$, being $T > 0$ a fixed final time. 
The strong formulation of the FK model reads: Find $c = c(\boldsymbol{x},t)$ such that
\begin{equation}
\begin{aligned}
    \frac{\partial c}{\partial t} & = \nabla \cdot (\mathbf{D} \nabla c) + \alpha c (1-c) + f
        &&\quad \text{in } \Omega \times (0,T], \\[4pt]
    (\mathbf{D}\nabla c) \cdot \boldsymbol{n} & = \phi_{N}
        &&\quad \text{on } \Gamma_{N} \times (0,T], \\[4pt]
    c & = c_{D}
        &&\quad \text{on } \Gamma_{D} \times (0,T], \\[4pt]
    c(\cdot,0) & = c_{0}
        &&\quad \text{in } \Omega.
\end{aligned}
\label{eq:eqFK}
\end{equation}
In Equation (\ref{eq:eqFK}), the reaction parameter $\alpha>0$ represents the local conversion rate and $\mathbf{D}$ is the diffusion tensor.
The forcing term $f = f(\boldsymbol{x},t)$ models the external addition/removal of mass. The Neumann condition prescribes a sufficiently regular flux $\phi_N = \phi_N (\boldsymbol{x},t)$ on the boundary of the domain denoted by $\Gamma_N$, while the Dirichlet condition imposes a concentration value $c_D = c_D (\boldsymbol{x},t)$ on the boundary portion $\Gamma_D$. We note that $\Gamma_N \cup \Gamma_D = \partial \Omega$ and $\Gamma_N \cap \Gamma_D = \emptyset$.
\par
We employ a standard definition of the scalar product in $L^2({\Omega})$, represented by $(\cdot,\cdot)_\Omega$, and of its induced norm $\|\cdot\|$, whose definition can be extended component-wise to vector-valued and tensor-valued functions \cite{salsa_partial_2022}. We define  $W_{D} = \left\{w \in H^{1}(\Omega) : w|_{\Gamma_{D}} = c_{D}\right\}$ and $W_{0} = H^{1}_{\Gamma_D}(\Omega)$; when $\Gamma_D = \emptyset$ and $\Gamma_N = \partial \Omega$, we have that $W_D = W_0 = H^1({\Omega})$.
We  introduce the following assumption on the data.
\begin{assumption} [Regularity of data]
We assume the following regularity for the data appearing in Equation~\eqref{eq:eqFK}:
\begin{itemize}
    \item $\mathbf{D} \in L^\infty(\Omega, \mathbb{R}^{d\times d})$ and $\exists\,
    d_0 > 0$ such that: $d_0 |\boldsymbol{\psi}|^2 \leq \boldsymbol{\psi}^T \mathbf{D} \boldsymbol{\psi} \quad \forall \boldsymbol{\psi} \in \mathbb{R}^d.$
    \item $ f \in L^2((0,T]; L^2(\Omega))$.
    \item $ \phi_N \in L^2((0,T]; L^2(\Gamma_N))$.
    \item $ c_D \in L^2((0,T]; H^{1/2}(\Gamma_D))$.
    \item $ c_0 \in L^2(\Omega)$.
\end{itemize}
\label{ass:reg_coeff}
\end{assumption}
\par
Setting
\begin{alignat}{3}
    & a(c,w)  = (\sqrt{\mathbf{D}}\nabla c, \sqrt{\mathbf{D}}\nabla w)_{\Omega}. 
    &&\qquad F(w)  = (f,w)_{\Omega} + (\phi_{N},w)_{\Gamma_{N}}
    &&\qquad \forall c \in W_D, \forall w\in W_0,
 \end{alignat}    
the weak formulation of problem \eqref{eq:eqFK} reads: Find $c = c(t) \in W_{D}$ such that for any $t \in (0,T]$:
\begin{equation}
\label{weakFK}
\begin{aligned}
     \left(\frac{\partial c}{\partial t}, w \right)_{\Omega} + a(c,w) - \left(\alpha c, w\right)_\Omega + \left(\alpha c^2, w\right)_\Omega & = F(w) &&\quad \forall w \in W_{0}, \\
     c(\cdot,0) & = {c}_{0} &&\quad \text{in} \ \Omega.
\end{aligned}
\end{equation}

\begin{remark}
\textit{Under Assumption~\ref{ass:reg_coeff}, if $f=0,\  \phi_N=0$ and $\Gamma_D = \emptyset$ with $c_0(\boldsymbol{x}) \in [0,1] \ \textrm{a.e.\,} \boldsymbol{x} \in \Omega$, $c_0(\boldsymbol{x})$ not identically equal to zero, it can be proved that the FK equation admits a travelling wave solution \cite{wen-shan_exact_2006}. Moreover, we can state that the solution $c(\boldsymbol{x},t) \in [0,1] \ \textrm{a.e.\,} \boldsymbol{x} \in \Omega, \ t > 0$.
This translates into the existence of two equilibrium points: an unstable equilibrium  at $c=0$ and a stable one at $c=1$, such that: }
\begin{center}
   $\displaystyle\lim_{t \to +\infty} c(\boldsymbol{x},t) = 1.$
\end{center}
\label{remark:FK_solution_non_negative}
\end{remark}

In view of the forthcoming discretisation, we next introduce the following the stabilised FK model.
\subsection{The stabilised FK model}
\label{sec:MFK}
In this section, we propose a stabilised version of the FK model (SFK) based on a modification of the reaction term. The strong formulation of the SFK model reads:
\begin{equation}
\begin{aligned}
    \frac{\partial \hat{c}}{\partial t} & = \nabla \cdot (\mathbf{D} \nabla \hat{c}) + \alpha |\hat{c}| (1-\hat{c}) + f
        &&\quad \text{in } \Omega \times (0,T], \\[4pt]
    (\mathbf{D}\nabla \hat{c}) \cdot \boldsymbol{n} & = \phi_{N}
        &&\quad \text{on } \Gamma_{N} \times (0,T], \\[4pt]
    \hat{c} & = c_{D}
        &&\quad \text{on } \Gamma_{D} \times (0,T], \\[4pt]
    \hat{c}(\cdot,0) & = c_{0}
        &&\quad \text{in } \Omega.
\end{aligned}
\label{eqMFK}
\end{equation}
In Equation~\eqref{eqMFK}, we recover the same coefficients of Equation~\eqref{eq:eqFK} and Assumption~\ref{ass:reg_coeff} still holds for the SFK model. 
\begin{remark}
As it will be made clear in the following sections, the choice of $|\hat{c}|$ in the reaction term of \eqref{eqMFK} transforms the unstable equilibrium $c=0$ into a stable one out of the region of interest, preventing the (numerical) solution from developing instabilities and unphysical oscillations. 
At the continuous level, the solutions of the classical FK model and of the SFK model coincide and strong consistency is guaranteed whenever $c(\boldsymbol{x},t) \ge 0$ and $\hat{c}(\boldsymbol{x},t) \ge 0$, respectively, e.g. as in the case of Remark~\ref{remark:FK_solution_non_negative}. At the discrete level, the proposed modification  ensures an automatic stabilization of the original formulation \eqref{eq_semidiscreteFK}. 
\eremk
\end{remark}

\section{Polytopal discontinuous Galerkin semi-discrete formulation}
\label{sec:PolyDG}
In this section,  we construct the space approximation for problems \eqref{weakFK} and \eqref{eqMFK} using the PolyDG method \cite{corti_discontinuous_2023}. 
To ease the discussion and without loss of generality, here and in the following we assume, that the Dirichlet datum is $c_D=0$. First, we introduce a mesh partition $\mathscr{T}_{h}$ of the domain $\Omega$ made of disjoint polygonal/polyhedral elements $K$. For each of them, we define its measure $|K|$ and its diameter $h_K$ such that $h = \max_{K \in \mathscr{T}_{h}} \{h_K\} < 1$. 
From now on, we will use the following notation $x \lesssim y$ to say that $\exists \, C > 0: x \leq Cy$, where the constant $C$ depends on the model parameters and on the polynomial degree $p$ but is independent of the mesh size $h$. Following \cite{cangiani_hp-version_2017}, the interface is defined as the intersection of the ($d - 1$)-dimensional facets of two neighbouring elements:
\begin{itemize}
    \item for $d=2$, the interfaces are line segments and the set of segments is denoted with $\mathscr{F}_{h}$;
    \item for $d=3$, any interface is a generic polygon and we assume that each of them can be decomposed into planar triangles, whose set is denoted with $\mathscr{F}_{h}$.
\end{itemize}
The set $\mathscr{F}_{h}$ can be seen as the union of interior faces ($\mathscr{F}_{h}^I$) and exterior faces ($\mathscr{F}_{h}^B$) lying on the domain boundary $\partial \Omega$. The boundary faces can be split taking into account the type of boundary condition:  $\mathscr{F}_{h}^B = \mathscr{F}_{h}^D \cup \mathscr{F}_{h}^N$, where $\mathscr{F}_{h}^D$ and $\mathscr{F}_{h}^N$ are the boundary faces contained in $\Gamma_D$ and $\Gamma_N$, respectively.
\par
We introduce 
the following mesh regularity assumption that will be exploited in 
the theory developed in the next sections.

\begin{assumption}[polytopal-regular mesh \cite{cangiani_hp-version_2017}]
\label{ass:mesh_cangiani}
The mesh sequence $\{\mathscr{T}_{h}\}_h$ uniformly satisfies the following properties:
\begin{enumerate}
    \item \textbf{polytopal-regularity}: let $\{S_K^F\}_{F\subset\partial K}$ be a set of non-overlapping $d$-dimensional simplices contained in $K$, then:
    \begin{equation*}
    \forall K\in\mathscr{T}_h\quad \exists \{S_K^F\}_{F\subset\partial K} \;\mathrm{such}\;\mathrm{that} \quad \forall F\subset\partial K \quad \overline{F}=\partial\overline{K}\cap\overline{S^F_K}\;\mathrm{and}\; h_K \lesssim d|S_K^F|\;|F|^{-1}.
    \end{equation*}
    \item \textbf{Covering mesh condition}: For each $\mathscr{T}_h \in \{\mathscr{T}_h\}_h$ there exists a shape-regular, simplicial covering $\hat{\mathscr{T}_h}$ such that for each pair $K\in\mathscr{T}_h$ and $\hat{K}\in\hat{\mathscr{T}_h}$ with $\hat{K}\subset K$ it holds $h_K\lesssim h_{\hat{K}}$ and $\max_{K\in\mathscr{T}_h}\{|K'|\}\lesssim 1$ where $K'\in\mathscr{T}_h: K'\cap \hat{K} \neq 0, \hat{K}\in\hat{\mathscr{T}_h}, K\subset \hat{K}$.
    \item \textbf{Local bounded variation property}: $\forall F\in\faces \, F\subset\partial K_1\cap\partial K_2 \quad K_1,K_2\in\mathscr{T}_h \Rightarrow h_{K_1}\lesssim h_{K_2} \lesssim h_{K_1}$ where the hidden constants are independent of the number of faces of $K_1$ and $K_2$.
\end{enumerate}
\end{assumption}
\begin{remark}
We remark that Assumption~\ref{ass:mesh_cangiani} is less restrictive than the classical choice of \cite{dipietro_hybrid_2020}, which requires stronger hypotheses on the size of the element faces.
\eremk
\end{remark}
\par
We define $\mathbb{P}_p(K)$ to be the space of polynomials of total degree $p \geq 1$ over a mesh element $K$. Then, we introduce the following discontinuous finite element space $ \Wh = \left\{w \in L^{2}(\Omega): w|_{K} \in \mathbb{P}_{p}(K) \ \  \forall K \in \mathscr{T}_{h} \right\}$. Let $F \in \mathscr{F}_{h}$ be a face shared by the elements $K^{\pm}$, then $\boldsymbol{n}^\pm$ is the corresponding unit normal vector on face $F$ pointing exterior to $K^{\pm}$. For sufficiently regular scalar-valued functions $v$ and vector-valued functions $\boldsymbol{q}$, we define the following trace operators \cite{arnold_unified_2002}, respectively:
\begin{itemize}
    \item jump operator $\jump{\cdot}$ on $F \in \mathscr{F}_{h}^I $: $\jump{v} = v^+ \boldsymbol{n}^+ + v^- \boldsymbol{n}^-, \quad \jump{\boldsymbol{q}} = \boldsymbol{q}^+ \cdot \boldsymbol{n}^+ +  \boldsymbol{q}^- \cdot \boldsymbol{n}^-$;
    \item average operator $\avg{\cdot}$ on $F \in \mathscr{F}_{h}^I $: $\avg{v} = \frac{1}{2} (v^+ + v^-), \quad \avg{\boldsymbol{q}} = \frac{1}{2}(\boldsymbol{q}^+ + \boldsymbol{q}^-)$,
\end{itemize}
where the superscripts $\pm$ refer to the traces of the functions on $F$ taken within the interior to $K^\pm$ respectively. 
Moreover, we can introduce similar operators on the face $F \in \mathscr{F}_{h}^D$ associated with the cell $K \in \mathscr{T}_{h}$ with $\boldsymbol{n}$ outward unit normal on $\partial \Omega$:
\begin{itemize}
    \item jump operator $\jump{\cdot}$ on $F \in \mathscr{F}_{h}^D$: $\jump{v} = v\boldsymbol{n}, \quad \jump{\boldsymbol{q}} = \boldsymbol{q} \cdot \boldsymbol{n}$;
    \item average operator $\avg{\cdot}$ on $F \in \mathscr{F}_{h}^D$: $\avg{v} = v, \quad\avg{\boldsymbol{q}} = \boldsymbol{q}$.
\end{itemize}
Additionally, we introduce the following broken Sobolev spaces $H^r(\mathscr{T}_{h}) = \left\{w_h \in L^2(\Omega): w_h|_K \in H^r(K) \ \forall K \in \mathscr{T}_{h}\right\}$ for $r \geq 1$ and the shorthand notation for the $L^2$-norm on a set of faces $\mathscr{F}$ as $\|\cdot\|_{\mathscr{F}}^2 = \sum_{F \in \mathscr{F}} \|\cdot\|_{L^2(F)}^2$.\\
Finally,  we define the penalization function $\eta: \mathscr{F}_{h} \rightarrow \mathbb{R}^+$:
\begin{equation}
\eta  = \eta_0
\begin{dcases}
     \displaystyle\frac{p^2}{\left\{ h\right\}_\mathrm{H}} &\text{on} \ F \in \mathscr{F}_{h}^I, \\
     \displaystyle\frac{p^2}{h}  &\text{on} \ F \in \mathscr{F}_{h}^D,
\end{dcases}
\label{penalty}
\end{equation}
where the operator $\left\{\cdot \right\}_\mathrm{H}$ stands for the harmonic average on $F \in \mathscr{F}_{h} ^I $ and $\eta_0$ is a suitable parameter, which has to be large enough to achieve stability.
\subsection{Interior-penalty DG semi-discrete formulation of the FK and SFK models} 
We define the bilinear form $\mathscr{A}: \Wh \times \Wh  \rightarrow \mathbb{R}$ as:
\begin{equation}
\label{eq:bilform}
    \mathscr{A}(c,w) =
        \displaystyle\int_{\Omega} (\mathbf{D}\nabla_{h} c) \cdot \nabla_{h} w
         + \int_{\mathcal{F}_{h}^{I} \cup \mathcal{F}_{h}^{D}} \Big(\eta \jump{c} \cdot \jump{w} - \avg{\mathbf{D}\nabla c} \cdot \jump{w} -  \jump{c} \cdot \avg{\mathbf{D} \nabla  w}\Big) d\sigma \quad \forall c,w \in \Wh,
\end{equation}
where $\nabla_h \cdot$ is the elementwise gradient \cite{arnold_unified_2002}.
Then, the semi-discrete PolyDG discretisation of the FK model \eqref{eq:eqFK} reads: Find $c_{h}(t) \in \Wh$ such that $\forall t > 0$:
\begin{equation}
\begin{aligned}
    \left(\frac{\partial c_{h}(t)}{\partial t}, w_{h} \right)_{\Omega} + \mathscr{A}(c_{h}(t),w_{h}) - (\alpha c_{h}(t),w_{h})_\Omega + (\alpha c_{h}^2(t),w_{h})_\Omega & = F(w_{h}) &&\quad \forall w_{h} \in \Wh, \\
    c_{h}(0) & = c_{h}^{0} &&\quad \text{in} \ \Omega_{h},
\end{aligned}
\label{eq_semidiscreteFK}
\end{equation}
where $c_h^0 \in \Wh$ is a suitable approximation of $c_0$. For the analysis of the method in \eqref{eq_semidiscreteFK}, we refer to \cite{corti_discontinuous_2023}.\\

Analogously, for the SFK model \eqref{eqMFK}, the resulting semi-discrete PolyDG formulation of the SFK equation reads: Find $\hat{c}_{h}(t) \in \Wh$ such that  $\forall t >0$:
\begin{equation}
\begin{aligned}
    \left(\frac{\partial \hat{c}_{h}(t)}{\partial t}, w_{h} \right)_{\Omega} + \mathscr{A}(\hat{c}_{h}(t),w_{h}) - (\alpha |\hat{c}_{h}|(t),w_{h})_\Omega + (\alpha |\hat{c}_{h}(t)|\hat{c}_h(t),w_{h})_\Omega & = F(w_{h}) &&\quad \forall w_{h} \in \Wh, \\
    \hat{c}_{h}(0) & = c_{h}^{0} &&\quad \text{in} \ \Omega_{h}.
\end{aligned}
\label{eq_semidiscreteMFK}
\end{equation}
Before presenting the analysis of formulation \eqref{eq_semidiscreteMFK}, we recall and state some preliminary estimates and properties.
\subsection{Preliminary estimates and properties}
First of all let us introduce the following DG-norms:
\begin{alignat}{3}
\label{eq:dgnorm}
    \dgnorm{v}^2 & = \|\sqrt{\mathbf{D}} \nabla_h{v}\|^2 + \|\sqrt{\eta} \jump{v}\|_{\mathscr{F}_{h}^I \cup \mathscr{F}_{h}^D}^2, &&\quad \forall v \in H^1({\mathscr{T}_{h}}), \\
    \label{eq:3DGnorm}
    \tdgnorm{v}^2 & = \dgnorm{v}^2 + \|\eta^{-\frac{1}{2}} \avg{\mathbf{D} \nabla_h v}\|_{\mathscr{F}_{h}^I \cup \mathscr{F}_{h}^D}^2, && \quad \forall v \in H^2(\mathscr{T}_{h}).
\end{alignat}
Under Assumption~\ref{ass:mesh_cangiani}, we recall that $\tdgnorm{v}^2 \lesssim \dgnorm{v}^2$ for any $v \in \Wh$, thanks to the trace-inverse inequality \cite{cangiani_hp-version_2017}. We introduce the following proposition on the continuity and coercivity of the bilinear form \cite{cangiani_hp-version_2014,cangiani_hp-version_2017}.
\begin{proposition}[Continuity and coercivity of $\mathscr{A}(\cdot,\cdot)$ \cite{cangiani_hp-version_2017}]
Let Assumption~\ref{ass:mesh_cangiani} be satisfied, then the bilinear form $\mathscr{A}(\cdot,\cdot)$ is continuous and coercive (provided that the penalty parameter $\eta$ is large enough):
\begin{alignat}{3}
\label{continuity}
|\mathscr{A}(v_h,w_h)| & \lesssim \dgnorm{v_h} \dgnorm{w_h}, && \quad \forall v_h, w_h \in \Wh, \\
\label{continuity2}
|\mathscr{A}(v_h,w_h)| & \lesssim \dgnorm{v_h} \tdgnorm{w}, && \quad \forall v_h \in \Wh,\,\forall w \in H^2(\mathcal{T}_h), \\
\label{coercivity}
\exists \mu>0 \text{ such that } \mu \dgnorm{v_h}^2 & \leq \mathscr{A}(v_h,v_h) &&\quad \forall v_h \in \Wh.
\end{alignat}
\end{proposition}
Then, for a polytopal mesh $\mathcal{T}_h$ which satisfies Assumption~\ref{ass:mesh_cangiani} and covering simplicial mesh $\mathcal{T}_h^\star$ as in \cite[Def. 4.27]{cangiani_hp_2022}, such that $K\subset K^\star\in\mathcal{T}_h^\star$ for any $K\in\mathcal{T}_h$, we can define a Stein operator $\mathcal{E}:H^m(K^\star)\rightarrow H^m(\mathbb{R}^d)$ for any $m\in\mathbb{N}_0$ such that: $\mathcal{E}v|_K = v$, $\|\mathcal{E}v\|_{H^m(\mathbb{R}^d)} \lesssim \|v\|_{H^m(K)}$ for all $v\in H^m(K)$.
\begin{proposition}[Polynomial interpolant \cite{cangiani_hp-version_2017}] \label{propInt}
Let Assumption~\ref{ass:mesh_cangiani} be satisfied; then the following estimate holds: $\forall u \in H^n(\mathscr{T}_h)$, $n > 1 + d/2$, $\exists u_I \in \Wh$ defined as in \cite{cangiani_hp-version_2017}, equation (3.25), such that
\begin{align}
     \tdgnorm{u-u_I}^2 \lesssim \ \displaystyle\sum_{K^\star\in \mathscr{T}_h} h_K^{2  \text{min}\left\{p+1,n\right\} - 2} \|\mathcal{E} u\|_{H^n(K^\star)}^2.
    \label{interp}
\end{align}
\end{proposition}
\color{black}
\section{Stability analysis of the SFK semi-discrete formulation}
\label{sec:stability}
This section aims to derive a stability analysis for the solution of the PolyDG semi-discrete  formulation \eqref{eq_semidiscreteMFK} of the SFK model \eqref{eqMFK}. 
Throughout the section, to simplify the presentation, we assume homogeneous mixed boundary conditions ($\phi_N=0$ and $c_D=0$).
\begin{theorem}[Stability result]
\label{thm:stability_estimate}
Let Assumptions~\ref{ass:reg_coeff} and \ref{ass:mesh_cangiani} be satisfied. For a sufficiently large penalty coefficient $\eta$, let $\hat{c}_h(t)$ be the solution of Equation~\eqref{eq_semidiscreteMFK} for any $t \in (0,T]$ and homogeneous mixed boundary conditions ($\phi_N=0$ and $c_D=0$). Then, the following estimates hold:
\begin{equation}
\label{eq:stability_estimate}
    \|\hat c_h(t)\|^2 + \int_0^t \dgnorm{\hat c_h(s)}^2 \ \ds + \int_0^t \|\hat c_h(s)\|^3_{L^3(\Omega)} \ \ds \lesssim \|\hat{c}_h^0\|^2 + |\Omega|t + \displaystyle\int_0^t \|f(s)\|^2 \ds,
\end{equation}
\begin{equation}
\label{eq:stability_estimate2}
\dgnorm{\hat{c}_h(t)}^2
+ \|\hat{c}_h(t)\|_{L^3(\Omega)}^3
\lesssim \|\hat{c}_h^0\|_\mathrm{DG}^2  + \|\hat{c}_{h}^0\|_{L^3(\Omega)}^3 + \|\hat{c}_h^0\|^2t + |\Omega|t^2 + (1+t) \int_0^t \|f(s)\|^2 \ds.
\end{equation}
\end{theorem}
Before proving Theorem~\ref{thm:stability_estimate}, we state the following remarks
\begin{remark}
Under the assumption of $f=0$, the stability estimate \eqref{eq:stability_estimate} of Theorem~\ref{thm:stability_estimate} becomes:
\begin{equation}
    \|\hat c_h(t)\|^2 + \int_0^t \dgnorm{\hat c_h(s)}^2 \ \ds + \int_0^t \|\hat  c_h(s)\|^3_{L^3(\Omega)} \ \ds \lesssim \|\hat{c}_h^0\|^2 + |\Omega| t = \hat{C}_S(\hat{c}_h^0, |\Omega|,t),
    \label{eq::stab_estim_f0}
\end{equation}
while the stability estimate \eqref{eq:stability_estimate2} reduces to:
\begin{equation}
    \dgnorm{\hat{c}_h(t)}^2
    + \|\hat{c}_h(t)\|_{L^3(\Omega)}^3
    \lesssim \|\hat{c}_h^0\|_\mathrm{DG}^2  + \|\hat{c}_{h}^0\|_{L^3(\Omega)}^3 + t\|\hat{c}_h^0\|^2 + |\Omega|t^2 = \widetilde{C}_{S}(\hat{c}_h^0, |\Omega|,t).
    \label{eq::stab_estim2_f0}
\end{equation}
The definitions of $\hat{C}_S$ and $\widetilde{C}_{S}$ will be used in the following analysis.
\end{remark}
\begin{remark}
The stability estimates \eqref{eq:stability_estimate} and \eqref{eq:stability_estimate2} are established for general polygonal and polyhedral meshes satisfying Assumption~\ref{ass:mesh_cangiani}, in the spirit of~\cite{cangiani_hp-version_2017}. The changes in the proof derive from the use of the absolute value of $c$, that creates an $L^3(\Omega)$-norm positive contribution at the LHS. This represents a substantial advance over the stability analysis for the FK model in~\cite{corti_discontinuous_2023}, which relied on more restrictive assumptions on the mesh \cite{dipietro_hybrid_2020}. In addition, the stability result for the SFK model is global in time, and the dependence on the final time~$t$ in the energy norm is linear rather than exponential as in \cite{corti_discontinuous_2023}, ensuring significantly better control for large times. The analysis can be extended to full Neumann boundary conditions by using the Sobolev-Poincaré-Wirtinger inequality from \cite{botti_sobolev_2026}, without loss of generality.
\eremk
\end{remark}
\begin{proof}{(Proof of Theorem~\ref{thm:stability_estimate}). }\textbf{Proof of estimate \eqref{eq:stability_estimate}.}
Let us consider Equation~\eqref{eq_semidiscreteMFK} and choose $w_h = \hat{c}_h(s)$:
\begin{equation*}
\left(\displaystyle\frac{\partial \hat{c}_{h}}{\partial t}, \hat{c}_{h} \right)_{\Omega} + \mathscr{A}(\hat{c}_{h},\hat{c}_{h}) - \alpha(|\hat{c}_{h}|,\hat{c}_{h})_\Omega + \alpha(|\hat{c}_{h}|,\hat{c}_{h}^2)_\Omega = F(\hat{c}_{h}).
\end{equation*}
After an integration in time in $(0,t)$ and using the coercivity estimate in \eqref{coercivity} and the H\"{o}lder inequality on the forcing term we obtain:
\begin{align*}
     \|\hat{c}_h(t)\|^2 + \displaystyle\int_0^t \left(\|\hat{c}_h(s)\|_\mathrm{DG}^2 + \|\hat{c}_h(s)\|_{L^{3}(\Omega)}^3 \right) \ds & \lesssim \|\hat{c}_h^0\|^2 + \displaystyle\int_0^t \big(\|\hat{c}_h(s)\|^2 + \|f(s)\| \|\hat{c}_h(s)\| \big) \ds \\
     & \lesssim \|\hat{c}_h^0\|^2 + \displaystyle\int_0^t \big(\|\hat{c}_h(s)\|^2 + \|f(s)\|^2 \big) \ds,
\end{align*}
where in the last step we used the Young inequality. The latter step that is needed is to bound the term $\|\hat{c}_h(s)\|^2$ on the right hand side. The latter can be rewritten using both H\"{o}lder and Young's inequalities:
\begin{equation*}
\|\hat{c}_h\|^2 \leq \|\hat{c}_h\|^2_{L^3(\Omega)} |\Omega|^{\frac{1}{3}} \lesssim \|\hat{c}_h\|^3_{L^3(\Omega)} + |\Omega|,
\end{equation*}
where $|\Omega|$ denotes the measure of the domain. Substituting the result we obtain:
\begin{equation*}
      \|c_h(t)\|^2 + \int_0^t \dgnorm{c_h(s)}^2 \ \ds + \int_0^t \|c_h(s)\|^3_{L^3(\Omega)} \ \ds \lesssim \|\hat{c}_h^0\|^2 + |\Omega|t + \displaystyle\int_0^t \|f(s)\|^2 \ds.
\end{equation*}
\par
\noindent
\textbf{Proof of estimate \eqref{eq:stability_estimate2}.}
Let us consider Equation~\eqref{eq_semidiscreteMFK} and choose $w_h = \dot{\hat{c}}_h = \displaystyle\frac{\partial \hat{c}_h}{\partial t}$:
\begin{equation*}
\|\dot{\hat{c}}_{h}\|^2 + \mathscr{A}(\hat{c}_{h},\dot{\hat{c}}_{h}) - \alpha(|\hat{c}_{h}|,\dot{\hat{c}}_{h})_\Omega + \alpha(|\hat{c}_{h}|\hat{c}_h,\dot{\hat{c}}_{h})_\Omega = F(\dot{\hat{c}}_{h}).
\end{equation*}
We can observe that due to the symmetry of $\mathscr{A}$:
\begin{align*}
    \int_0^t \mathscr{A}(\hat{c}_{h}(s),\dot{\hat{c}}_{h}(s)) \ds & \,= \frac{1}{2}\mathscr{A}(\hat{c}_{h}(t),\hat{c}_{h}(t)) - \frac{1}{2}\mathscr{A}(\hat{c}_{h}^0,\hat{c}_{h}^0), \\
\int_0^t \alpha(|\hat{c}_{h}(s)|\hat{c}_h(s),\dot{\hat{c}}_{h}(s))_\Omega & \, = \frac{\alpha}{3}\|\hat{c}_{h}(t)\|_{L^3(\Omega)}^3 - \frac{\alpha}{3}\|\hat{c}_{h}^0\|_{L^3(\Omega)}^3.
\end{align*}
Then, after an integration in time in $(0,t)$ and using the coercivity and continuity estimates in \eqref{coercivity} and \eqref{continuity} we obtain:
\begin{align*}
     \|\hat{c}_h(t)\|_\mathrm{DG}^2 &  \,+
     \|\hat{c}_{h}(t)\|_{L^3(\Omega)}^3 +
     \int_0^t \|\dot{\hat{c}}_h(s)\|^2 \ds \\ & \,\lesssim \|\hat{c}_h^0\|_\mathrm{DG}^2  + \|\hat{c}_{h}^0\|_{L^3(\Omega)}^3 + \int_0^t (|\hat{c}_{h}(s)|,\dot{\hat{c}}_{h}(s))_\Omega \ds + \int_0^t (f(s),\dot{\hat{c}}_{h}(s))_\Omega \ds \\ & \, \lesssim \|\hat{c}_h^0\|_\mathrm{DG}^2  + \|\hat{c}_{h}^0\|_{L^3(\Omega)}^3 + \int_0^t \left(\|\hat{c}_{h}(s)\|^2 + \|f(s)\|^2 + \|\dot{\hat{c}}_{h}(s)\|^2\right) \ds,
\end{align*}
where in the last steps we used the H\"older and Young inequalities. After simplification of $\|\dot{\hat{c}}_h(s)\|^2$, we use the stability estimate \eqref{eq:stability_estimate} to bound $\|\hat{c}_h(s)\|$ on the right hand side:
\begin{equation*}
     \|\hat{c}_h(t)\|_\mathrm{DG}^2  + \|\hat{c}_{h}(t)\|_{L^3(\Omega)}^3  \lesssim \|\hat{c}_h^0\|_\mathrm{DG}^2  + \|\hat{c}_{h}^0\|_{L^3(\Omega)}^3 + t\|\hat{c}_h^0\|^2 + |\Omega|t^2 + (1+t) \int_0^t \|f(s)\|^2 \ds.
\end{equation*}
\end{proof}

\section{\textit{A-priori} error analysis of the SFK semi-discrete formulation}
\label{sec:errorAnalysis}
In this section, we derive an \textit{a-priori} error estimate for the solution of the semi-discrete PolyDG formulation \eqref{eq_semidiscreteMFK}. We assume a homogeneous forcing term $f=0$ and homogeneous mixed boundary conditions, i.e., $\phi_N=0$ and $c_D=0$, respectively.  
Before stating the main result of this section, we recall that, under the above assumptions on the data, at the continuous level (see \cite{smoller_shock_1994}), if the initial condition $c_0(\boldsymbol{x}) \in (0,1)$ and for homogeneous mixed boundary conditions $\phi_N=0$, $c_D=0$, and $f=0$, then for each $(\boldsymbol{x},t) \in \Omega\times(0,T]$, we get that:
\begin{equation}
    \|c(t)\|^2_{L^\infty(\Omega)} \leq 1 \quad \forall t \in (0,T)
\label{upperlimitc}
\end{equation}
Next, for any $t$, given $c \in H^n(\mathscr{T}_h)$, $n > 1 + d/2$, we defined the interpolant $c_I$ as in Proposition~\ref{propInt}. By construction,  $c_I$ belongs to the discrete space $\Wh$.
This guarantees that $c_I$ is bounded in the $L^\infty$-norm, that is $\exists M_I >0$ such that
\begin{equation}
\|c_I(t)\|_{L^{\infty}(\Omega)} \leq M_I
\qquad \forall t \in (0,T) ,
\label{upperlimitcI}
\end{equation}
with $M_I$ independent of $h$, provided that the mesh size is small enough, i.e. $h<1$.
Indeed, we have that, using the triangular inequality $c_I = (c_I - c) + c$, \cite[Lemma 20]{cangiani_hp-version_2017}, and \eqref{upperlimitc}, we obtain
\begin{align*}
     \|c_I\|_{L^{\infty}(\Omega)} & \leq \|c\|_{L^{\infty}(\Omega)} + \|c_I-c\|_{L^{\infty}(\Omega)} \\
     & = \|c\|_{L^{\infty}(\Omega)} + \     \underset{K\in\mathcal{T}_h}{\mathrm{max}}\left\{\|c_I-c\|_{L^{\infty}(K)}\right\} \\
     &  \leq \|c\|_{L^{\infty}(\Omega)} + \     \underset{K\in\mathcal{T}_h}{\mathrm{max}}\left\{\|c_I-c\|_{L^{\infty}(K^\star)}\right\} \\
     & \lesssim 1 + \underset{K\in\mathcal{T}_h}{\mathrm{max}} \left\{h_{K^\star}^{\text{min}\left\{p+1,n\right\} - d/2} \|\mathcal{E} c\|_{H^n(K^\star)}\right\} \leq M_I,
\end{align*}
under the assumption that $h \leq 1$.

We are now ready to state the main result of this section.
\begin{theorem}
\label{th:a-priori-error-bound}
Let us consider problem \eqref{eqMFK} with $f=0$, $\phi_N = 0$, and $c_D = 0$.
Let Assumption~\ref{ass:reg_coeff} be satisfied and let $\hat{c}$ be the weak solution of \eqref{eqMFK} for any $t \in (0,T]$.
Assume that it satisfies the following additional regularity:
\begin{equation}    
    \hat{c} \in C^1([0,T]; H^n(\Omega) \cap L^\infty(\Omega)),
\label{addregul}
\end{equation}
where $n > 1 + d/2$. Let us also assume further regularity on the initial condition $c_0 \in W_0$ and such that $c_0(\boldsymbol{x})\in[0,1]$ for any $\boldsymbol{x}\in\Omega$. Let Assumption~\ref{ass:mesh_cangiani} be satisfied. For a sufficiently large penalty coefficient $\eta$, let $\hat{c}_h$ be the solution of \eqref{eq_semidiscreteMFK} for any $t \in (0,T]$. Then, for any $t\in (0,T]$, it holds:
\begin{equation}
\begin{split}
\|\hat{c}(t) - \hat{c}_h(t)\|^2 + \int_0^t \tdgnorm{\hat{c}(s) - \hat{c}_h(s)}^2 \ds &
    \lesssim \sum_{K \in \mathscr{T}_{h}} h_K^{2 \text{min} \left\{p+1, n\right\} -2} \|\hat{c}(t)\|_{H^n(K)}^2 \\ + & \sum_{K \in \mathscr{T}_{h}} h_K^{2 \text{min} \left\{p+1, n\right\} -2} \int_0^t \left(\|\dot{\hat{c}}(s)\|_{H^n(K)}^2 + \|\hat{c}(s)\|_{H^n(K)}^2 \right) \ds,
\end{split} 
    \label{estimatedglobalerror}
\end{equation}
provided $\mu \gtrsim \alpha$, where the hidden constant depends on the interpolation constant $M_I$ of Equation~\eqref{upperlimitcI}, the stability constant $\widetilde{C}_S$ defined in Equation~\eqref{eq::stab_estim2_f0} and the discrete Sobolev embedding constant $C_{E_q}$ for the $L^q(\Omega)$ spaces in \cite{botti_sobolev_2026}.
\label{thmerror}
\end{theorem}
\begin{proof}
We subtract Equation~\eqref{eq_semidiscreteMFK} from the weak formulation of \eqref{eqMFK} to obtain:
\begin{equation*}
    \left(\dot{\hat{c}} - \dot{\hat{c}}_{h}, w_{h}\right)_{\Omega} + \mathscr{A} \left(\hat{c} - \hat{c}_{h},w_{h}\right) - \alpha\left(|\hat{c}| - |\hat{c}_{h}|,w_{h}\right)_{\Omega} + \alpha\left(\hat{c} |\hat{c}| - \hat{c}_{h} |\hat{c}_{h}|,w_{h}\right)_{\Omega} = 0 \quad \forall w_{h} \in \Wh.
\end{equation*}
\noindent Then, we define the errors $ \hat{e}_{h}^{c} = c_{I} - \hat{c}_{h}$ and $e_{I}^{c} = \hat{c} - c_{I}$, with $c_{I}$ being a suitable interpolant.
By setting $w_h = \hat{e}_{h}^{c} $:
\begin{equation*}
    \left(\dot{\hat{e}}_{h}^{c},\hat{e}_{h}^{c}\right)_{\Omega} + \mathscr{A}\left(\hat{e}_{h}^{c}, \hat{e}_{h}^{c}\right) - \alpha\left(|\hat{c}| - |\hat{c}_h|, \hat{e}_{h}^{c}\right)_{\Omega} + \alpha \left(\hat{c} |\hat{c}| - \hat{c}_{h} |\hat{c}_{h}|, \hat{e}_{h}^{c}\right)_{\Omega} = -\left(\dot{e}_{I}^{c},\hat{e}_{h}^{c}\right)_{\Omega} - \mathscr{A}\left(e_{I}^{c}, \hat{e}_{h}^{c}\right).
\end{equation*}
We remark that $\hat{e}_{h}^{c}(0) = 0$ by setting $ \hat{c}_{h}^{0} = c_I(0)$ and, by using the symmetry of the scalar product,
\begin{equation*}
\left(\dot{\hat{e}}_{h}^{c},\hat{e}_{h}^{c}\right)_{\Omega} = \frac{1}{2} \displaystyle\frac{\mathrm{d}}{\mathrm{d}t} \left( \hat{e}_{h}^{c}, \hat{e}_{h}^{c}\right)_{\Omega} = \frac{1}{2} \frac{\mathrm{d}}{\mathrm{d}t} \| \hat{e}_{h}^{c}\|^2.
\end{equation*}
We integrate in time between $0$ and a generic time instant $t$ such that:
\begin{align*}
\|\hat{e}_h^{c}(t)\|^{2}
&\;+ \int_{0}^{t} \mu \dgnorm{\hat{e}_{h}^{c}(s)}^{2} \ \ds
\lesssim
\int_{0}^{t} |(\dot{e}_I^c (s), \hat{e}_h^c (s))_\Omega| \ds - \int_{0}^{t} |\mathscr{A}(e_I^c(s), \hat{e}_h^c(s))| \ds \\
&\;+ \int_{0}^{t} \alpha\left( |\hat{c}(s)| - |\hat{c}_h(s)|,\hat{e}_h^{c}(s)\right)_\Omega \ds
+ \int_{0}^{t} \alpha \left( \hat{c}(s)|\hat{c}(s)| - \hat{c}_h(s) |\hat{c}_h(s)|, \hat{e}_h^c(s)\right)_{\Omega} \ds
\end{align*}
Now, we can use Equation~\eqref{continuity} and H\"{o}lder inequality to obtain
\begin{align}
\nonumber  \|\hat{e}_h^{c}(t)\|^{2}
&\;+ \int_{0}^{t} \mu \dgnorm{\hat{e}_{h}^{c}(s)}^{2} \ \ds
\lesssim
\int_{0}^{t} \|\dot{e}_I^c (s)\|\,\|\hat{e}_h^c (s)\| \ds + \int_{0}^{t} \tdgnorm{e_I^c(s)}\,\dgnorm{\hat{e}_h^c(s)} \ds \\
&\;+ \int_{0}^{t} \alpha\underbrace{\left( |\hat{c}(s)| - |\hat{c}_h(s)|,\hat{e}_h^{c}(s)\right)_\Omega}_{(1)} \ds
+ \int_{0}^{t} \alpha \underbrace{\left( \hat{c}(s)|\hat{c}(s)| - \hat{c}_h(s) |\hat{c}_h(s)|, \hat{e}_h^c(s)\right)_{\Omega}}_{(2)} \ds
\label{eq:partial_res}
\end{align}
The integrand (1) can be bounded by combining the triangular and Cauchy-Schwarz inequalities:
\begin{align*}
\left(|\hat{c}| - |\hat{c}_h|,\hat{e}_h^{c}\right)_\Omega
&\,= \int_{\Omega} (|\hat{c}| - |\hat{c}_h|) \hat{e}_h^c
\leq \int_{\Omega} \left| |\hat{c}| - |\hat{c}_h| \right|\, |\hat{e}_{h}^c|
\leq \int_{\Omega} |\hat{c} - \hat{c}_h |\, |\hat{e}_h^c|
\\ &\,= \int_{\Omega} |\underbrace{\hat{c} - c_I}_{= \  e_I^c} + \underbrace{c_I - \hat{c}_h}_{= \ \hat{e}_h^c}| \; |\hat{e}_h^c| \leq \int_{\Omega} (|e_I^c| + |\hat{e}_h^c|) |\hat{e}_h^c| \leq
 \|e_I^c\| \ \|\hat{e}_h^c\| + \|\hat{e}_h^c\|^2.
\end{align*}
Moreover, the integrand denoted by (2) can be manipulated as follows:
\begin{align*}
\hat{c}|\hat{c}| - \hat{c}_h |\hat{c}_h| & \, = \hat{c}|\hat{c}| - \hat{c}|c_I| + \hat{c}|c_I| - c_I |c_I| + c_I |c_I| - |c_I| \hat{c}_h + |c_I| \hat{c}_h -|\hat{c}_h| \hat{c}_h \\ & \, =
\underbrace{\hat{c}(|\hat{c}| - |c_I|)}_{\mathrm{(I)}} + \underbrace{|c_I| (\hat{c} - c_I)}_{\mathrm{(II)}} + \underbrace{|c_I|(c_I - \hat{c}_h)}_{\mathrm{(III)}} + \underbrace{\hat{c}_{h} (|c_I| - |\hat{c}_h|)}_{\mathrm{(IV)}}.
\end{align*}
Each resulting term can be treated in the following way:
\begin{itemize}
    \item The term (I) can be bounded using the triangular and Cauchy-Schwarz inequalities and the bound of continuous solution \eqref{upperlimitc}:
    \begin{equation*}
        |(\mathrm{I})| \leq \left|\left(\hat{c}(|\hat{c}|-|c_I|),\hat{e}_h^c \right)_{\Omega} \right| \leq  \|\hat{c}\|_{L^{\infty}(\Omega)} \|\hat{c}-c_I\| \, \|\hat{e}_{h}^{c}\| \leq \|{e}_{I}^{c}\| \, \|\hat{e}_{h}^{c}\|.
    \end{equation*}
    \item The term (II) can be bounded using the Cauchy-Schwarz inequality and the interpolant bound \eqref{upperlimitcI}:
    \begin{equation*}
        |(\mathrm{II})| \leq \left|\left(|c_I|(\hat{c}-c_I),\hat{e}_h^c \right)_{\Omega} \right| \leq  \|c_I\|_{L^{\infty}(\Omega)} \|\hat{c}-c_I\| \, \|\hat{e}_{h}^{c}\| \leq M_I \|{e}_{I}^{c}\| \, \|\hat{e}_{h}^{c}\|.
    \end{equation*}
    \item The term (III) can be bounded using the Cauchy-Schwarz inequality, the bound of interpolant \eqref{upperlimitcI}, and the Sobolev-Poincar\'e-Wirtinger inequality with constant $C_{E_2}$ \cite{botti_sobolev_2026}:
    \begin{equation*}
        |(\mathrm{III})| \leq \left|\left(|c_I|(c_I-\hat{c}_h),\hat{e}_h^c \right)_{\Omega} \right| \leq  \|c_I\|_{L^{\infty}(\Omega)} \|\hat{e}_{h}^{c}\|^2 \leq M_I C_{E_2} \dgnorm{\hat{e}_{h}^{c}}^2.
    \end{equation*}
    \item The term (IV) can be bounded using the stability estimate of the $L^3(\Omega)$-norm of the solution \eqref{eq::stab_estim2_f0}, the triangular, H\"{o}lder, and the Sobolev-Poincar\'e-Wirtinger inequality with constant $C_{E_3}$\cite{botti_sobolev_2026}:
    \begin{equation*}
        |(\mathrm{IV})|
        \leq \left|\left(\hat{c}_h(|c_I|-|\hat{c}_h|),\hat{e}_h^c \right)_{\Omega} \right|
        \leq \left|\left(\hat{c}_h,(\hat{e}_h^c)^2 \right)_{\Omega}\right|
        \leq \|\hat{c}_h\|_{L^3(\Omega)} \|\hat{e}_{h}^{c}\|^2_{L^3(\Omega)} \leq \widetilde{C}_{S}^{\frac{1}{3}} C_{E_3} \dgnorm{\hat{e}_{h}^{c}}^2.
    \end{equation*}
\end{itemize}
Then, Equation~\eqref{eq:partial_res} using the equivalence of the DG-norms in $\Wh$ and provided $\mu \gtrsim \alpha$ becomes:
\begin{equation*}
\|\hat{e}_h^{c}(t)\|^{2}
+ \int_{0}^{t} \tdgnorm{\hat{e}_{h}^{c}(s)}^{2} \ \ds
\;\lesssim
\int_{0}^{t} \left(\|\dot{e}_I^c (s)\|+ \|e_I^c(s)\|\right)\|\hat{e}_h^c (s)\| \ds + \int_{0}^{t} \tdgnorm{e_I^c(s)}\dgnorm{\hat{e}_h^c(s)} \ds.
\end{equation*}
Using once again H\"{o}lder inequality, Gr\"{o}nwall's lemma \cite{quarteroni_numerical_2017} and the estimate in \eqref{interp}, we obtain:
\begin{align*}
\|\hat{e}_h^{c}(t)\|^{2}
+ \int_{0}^{t} \tdgnorm{\hat{e}_{h}^{c}(s)}^{2} \ \ds &\; \lesssim
\int_{0}^{t} \left(\|\dot{e}_I^c (s)\|^2+\|e_I^c (s)\|^2+\dgnorm{e_I^c(s)}^2\right) \ds \\
&\; \lesssim \sum_{K \in \mathscr{T}_{h}} h_K^{2 \text{min} \left\{p+1, n\right\} -2} \int_0^t \left(\|\dot{\hat{c}}(s)\|_{H^n(K)}^2 + \|\hat{c}(s)\|_{H^n(K)}^2 \right) \ \ds.
\end{align*}
Finally, using the triangular inequality, which concludes the proof.
\end{proof}

\section{Fully-discrete SFK formulation}
\label{sec:discreteMFK}
Let $\left\{\varphi_j\right\}_{j=0}^{N-1}$ be a suitable basis for $\Wh$, where $N$ is the dimension of the space $\Wh$. Then, $\hat{c}_h(t) = \displaystyle\sum_{j=0}^{N-1} C_j(t) \varphi_j$ and we denote by $\mathbf{C} \in \mathbb{R}^N$ the corresponding vector of expansion coefficients in the chosen basis. 
We define the following components $\forall i,j = 0, ..., N-1$:
\begin{alignat*}{2}
    [\mathbf{M}]_{ij} &= (\varphi_j, \varphi_i)_\Omega 
    &&\quad (\text{Mass matrix}) \\
    [\mathbf{A}]_{ij} &= \mathscr{A}(\varphi_j, \varphi_i) 
    &&\quad (\text{Stiffness matrix}) \\[4pt]
    [\mathbf{R}(\mathbf{C}(t))]_{i} &= (\alpha |\hat{c}_h(t)|, \varphi_i)_\Omega 
    &&\quad (\text{non-linear reaction vector}) \\
    [\mathbf{\widetilde{R}} (\mathbf{C}(t))]_{ij} &= (\alpha |\hat{c}_h(t)| \varphi_j, \varphi_i)_\Omega 
    &&\quad (\text{non-linear reaction matrix})
\end{alignat*}
Moreover, we define the forcing term $[\mathbf{F}]_i = F(\varphi_i)$ $\forall i = 0, ..., N-1$. Then problem \eqref{eq_semidiscreteMFK} can be written as:
\begin{align}
    \mathbf{M}\dot{\mathbf{C}}(t) + \mathbf{A} \mathbf{C}(t) - \mathbf{R}(\mathbf{C}(t)) + \mathbf{\widetilde{R}}(\mathbf{C}(t)) \mathbf{C}(t) & = \mathbf{F}(t), \quad \forall t \in (0,T] \\
    \mathbf{C}(0) & = \mathbf{C}^0.
\label{fullydiscr}
\end{align}
We define a partition of the time interval $[0,T]$ into $N_T$ subintervals $[t_{n}, t_{n+1}]$ for $n=0, \dots, N_T - 1$, with time step $\Delta t = t_{n+1}-t_n$. We employ a semi-implicit Crank--Nicolson scheme. Given the initial conditions $\mathbf{C}^{-1},\mathbf{C}^0$, for $n=0, \dots, N_T - 1$, find $\mathbf{C}^{n+1} \approx \mathbf{C}(t_{n+1})$ such that:
\begin{equation}
    \left( \mathbf{M} + \frac{\Delta t}{2} \bigl(\mathbf{A} + \mathbf{\widetilde{R}}(\mathbf{C}^*)\bigr) \right) \mathbf{C}^{n+1} = \left( \mathbf{M} - \frac{\Delta t}{2} \bigl(\mathbf{A} + \mathbf{\widetilde{R}}(\mathbf{C}^*)\bigr) \right) \mathbf{C}^{n} + \Delta t \mathbf{R}(\mathbf{C}^*) + \frac{\Delta t}{2}(\mathbf{F}^{n} + \mathbf{F}^{n+1}),
\label{eq:fully_discrete_cn}
\end{equation}
where the semi-implicit extrapolation term $\mathbf{C}^*$ is defined as:
\begin{equation}
    \mathbf{C}^* = 
    \begin{cases}
        \mathbf{C}^0, & \text{for } n = 0, \\[4pt]
        \dfrac{3}{2} \mathbf{C}^n - \dfrac{1}{2} \mathbf{C}^{n-1}, & \text{for } n \ge 1.
    \end{cases}
\end{equation}

\section{Numerical results: verification of the SFK model}
\label{sec:verification}
In this section, we verify the theoretical results of the previous sections by presenting a convergence analysis based on a manufactured solution in both two-dimensional (2D) and three-dimensional (3D) settings.

\subsection{Test Case 1: verification in 2D}
We design a convergence test to verify the theoretical bounds derived in the previous sections. For the 2D numerical tests, we use the MATLAB library \texttt{lymph} \cite{lymph}, which implements the SFK solver on polytopal meshes. We consider a unit square domain $\Omega = (0,1)^2$ and the corresponding mesh has been generated with PolyMesher \cite{talischi_polymesher_2012}. In this test, the manufactured solution is given by:
\begin{equation*}
    \hat{c}(x,y,t) = \frac{1}{4}\left(\cos(\pi x) \cos(\pi y) + 2\right) e^{-t}.
\end{equation*}
To simplify the problem, we assume isotropic diffusion, i.e.\ $\mathbf{D} = \mathbf{I}$, which yields $\alpha = 1$ and $\eta_0 = 10$. The forcing term $f$ and the non-homogeneous Dirichlet boundary conditions are derived from the exact solution. We employ a fixed time-step $\Delta t = 10^{-5}$ and final time $T = 10^{-3}$ for both $h$- and $p$-convergence studies. For the $h$-convergence analysis, we consider four meshes with $N_{\mathrm{el}} = 30, 100, 300, 1000$ elements and polynomial degrees $p = 1, \ldots, 6$. For the $p$-convergence study, we fix the mesh with $N_{\mathrm{el}} = 30$ elements and vary the polynomial degree in the range $p = 1, \ldots, 8$.
\par
Figure~\ref{fig:hconv2D} displays the errors at final time $T$ as functions of the mesh size for polynomial degrees $p = 1, \ldots, 6$. The errors are measured in both the $L^2$-norm and the DG-norm. The observed convergence rates are consistent with Theorem~\ref{thmerror}: we observe an order-$p$ convergence rate in the DG-norm and order-$p+1$ in the $L^2$-norm. Although the latter is not covered by the present theoretical analysis, optimal rates are clearly observed in practice.
\par
Figure~\ref{fig:energy_pconv} (left) reports the errors computed in the energy functional defined as
\begin{equation}
    \enorm{\hat{c}_h(t)}^2 := \|\hat{c}_h(t)\|^2 + \int_0^t \dgnorm{\hat{c}_h(s)}^2 \ \ds + \int_0^t \|\hat{c}_h(s)\|^3_{L^3(\Omega)} \ \ds,
\label{eq:energyNorm}
\end{equation}
as in Theorem~\ref{th:a-priori-error-bound}, confirming the expected order-$p$ convergence under mesh refinement. On the right, we present the $p$-convergence results on the mesh with $N_{\mathrm{el}} = 30$, measured in the $L^2$-norm, DG-norm and energy functional. In this case, the errors exhibit exponential decay with respect to the polynomial degree $p$, indicating numerically optimal behaviour despite the lack of a corresponding theoretical estimate.
\begin{figure}[!t] 
    \centering   
    \begin{subfigure}[c]{0.48\textwidth}
        \centering
        \resizebox{\textwidth}{!}{\input{Images/h_2Dconv_L2}}
    \end{subfigure}
    \hfill
    \begin{subfigure}[c]{0.48\textwidth}
        \centering
        \resizebox{\textwidth}{!}{\input{Images/h_2Dconv_DG}}
    \end{subfigure}
    \caption{Test Case 1: computed errors in the $L^2$-norm (left) and DG-norm (right) versus the mesh-size $h$. }
    \label{fig:hconv2D}
\end{figure}
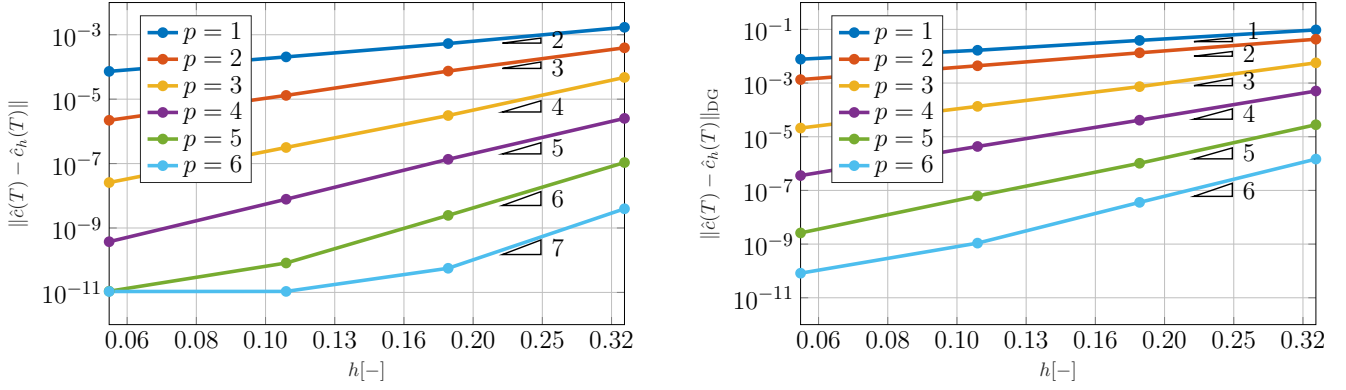
\begin{figure}[!t] 
    \centering
    \begin{subfigure}[c]{0.48\textwidth}
        \centering
        \resizebox{\textwidth}{!}{\input{Images/h_2Dconv_eps}}
    \end{subfigure}
    \hfill
    \begin{subfigure}[c]{0.48\textwidth}
        \centering
        \resizebox{\textwidth}{!}{\input{Images/p_convergence_2D}}
    \end{subfigure}
    \caption{Test Case 1: computed errors w.r.t the mesh size $h$ (left) and w.r.t. the polynomial $p$ (right).}
    \label{fig:energy_pconv}

\end{figure}
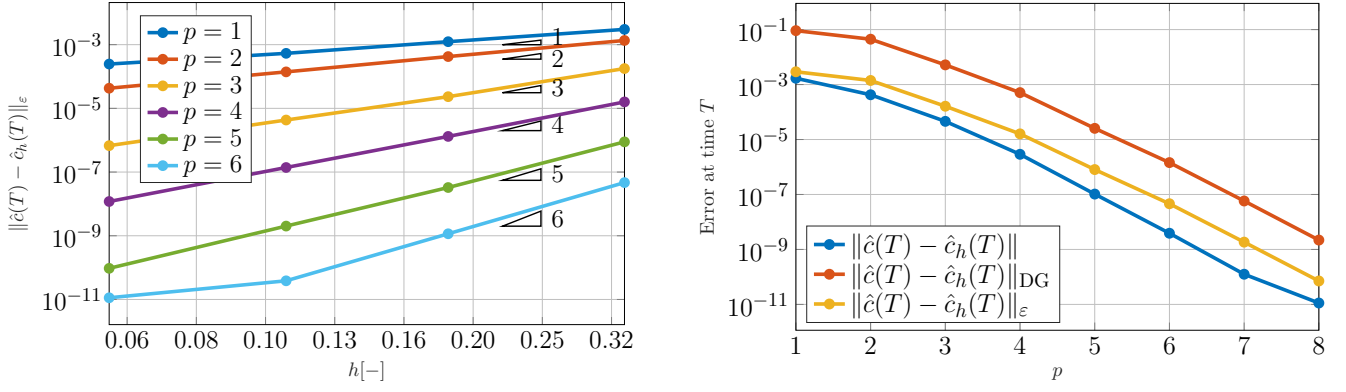
\FloatBarrier

\subsection{Test Case 2: verification in 3D}
\begin{figure}[!t] 
    \centering
    \begin{subfigure}[c]{0.48\textwidth}
        \centering
        \includegraphics[width=0.8\textwidth]{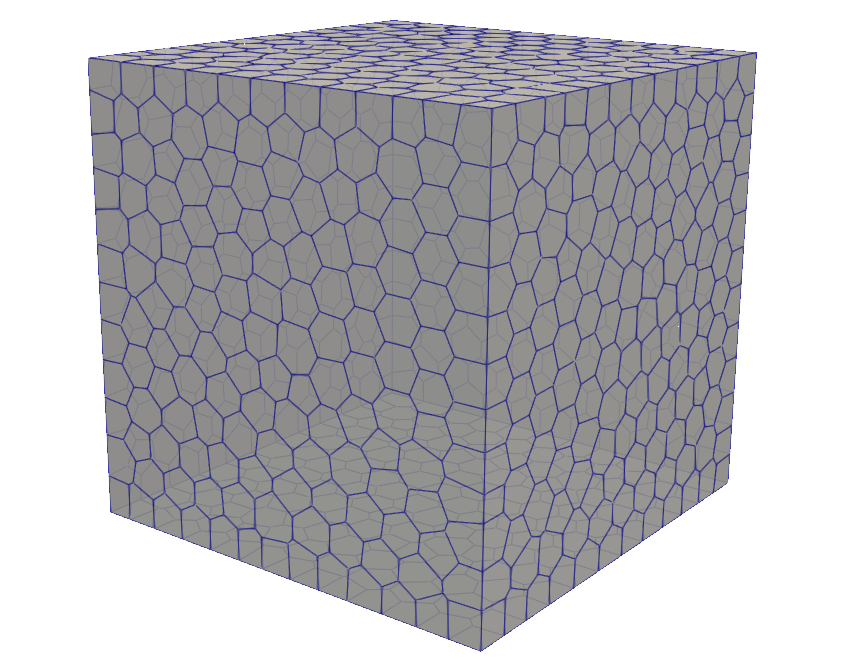}
    \end{subfigure}
    \hfill
    \begin{subfigure}[c]{0.48\textwidth}
        \centering
        \resizebox{\textwidth}{!}{\input{Images/p_convergence_3D}}
    \end{subfigure}
    \\
    \begin{subfigure}[c]{0.48\textwidth}
        \centering
        \resizebox{\textwidth}{!}{\input{Images/h_convergence_3D_L2}}
    \end{subfigure}
    \hfill
    \begin{subfigure}[c]{0.48\textwidth}
        \centering
        \resizebox{\textwidth}{!}{\input{Images/h_convergence_3D_DG}}
    \end{subfigure}

    \caption{Test Case 2: example of a Voronoi grid used (top-left) and computed errors with respect to different polynomial degrees (top-right). Computed errors and convergence rates in the $L^2$-norm (bottom-left) and DG-norm (bottom-right). }
    \label{fig:dg_and_pconv}
\end{figure}

In this section, we present some numerical results demonstrating the theoretical bounds in the 3D case. The simulations in this section have been carried out with \texttt{Vulpes} (Versatile and parallel computational Library for high-order discontinuous Polytopal Element Simulations) \cite{Vulpes}, an open-source and general-purpose C++ library developed for the numerical solution on polytopal meshes of multiphysics and multiscale differential problems. 
\par
We consider a cubic domain $\Omega = (0,1)^3$. The manufactured solution of the problem is given by:
\begin{equation*}
     \hat{c}(x,y,z,t) = (\sin(\pi x) \sin(\pi y) \sin(\pi z) + 2)e^{-t}.
\end{equation*}
\noindent We assume isotropic diffusion, which means $\mathbf{D} = \mathbf{I}$, $\alpha = 0.1$ and $\eta_0 = 10$. The forcing term $f$ and the Dirichlet boundary conditions on $\partial \Omega$ are derived according to the exact solution. We use a fixed time-step $\Delta t = 10^{-8}$ and a final time of simulation $T = 10^{-6}$ for both $h$-convergence and $p$-convergence studies. For the $h$-convergence study, we select five meshes with $h = 0.45, 0.35, 0.279, 0.226, 0.17479$ respectively generated via \texttt{Voro++} \cite{rycroft2009voro} and we repeat the simulation for polynomial degrees $p = 1, \dots, 4$. 
\par
The second row of Figure~\ref{fig:dg_and_pconv} shows the computed errors at final time $T$ versus the mesh size for each polynomial degree in the $L^2$-norm and DG-norm, respectively. We observe that the expected convergence rates are obtained, according to the results of Theorem~\ref{thmerror}.
\par
For the $p$-convergence study, we use different polynomial degrees $p = 1, ..., 6$ on the mesh with $h = 0.45$. Figure~\ref{fig:dg_and_pconv} (top-right) shows the results of this analysis. We notice an exponential decay of the errors with respect to the polynomial degree $p$, which confirms that optimal convergence is achieved. As already explained in the 2D case, there is no theoretical estimate that validates this experimental behaviour.
\FloatBarrier
\section{Numerical results: proof-of-concept in a realistic test cases and qualitative validation}
\label{sec:validation}
In this section, we first present a comparison of the FK and SFK models for a two-dimensional travelling wave solution. Next, we validate the proposed approach by simulating the propagation of $\alpha$-synuclein, a key pathogenic protein implicated in Parkinson's disease and other neurodegenerative diseases, demonstrating that the proposed scheme is stable, high-order accurate, and robust in a biologically relevant computational setting.
\subsection{Test case 3: comparison of the FK and SFK models for a 2D travelling wave solution}
\label{sec:test3}
We consider a solution of the form $c(x, y, t) = \psi(x - vt) = \psi(\xi)$. By substituting the expression into problem \eqref{eq:eqFK} with homogeneous forcing term $f=0$ and free boundary conditions, we obtain the following system of equations:
\begin{equation}
\label{eq:wave_problem}
\begin{aligned}
\chi'(\xi) & = -\dfrac{v}{d_{ext}} \chi(\xi) + \dfrac{1}{d_{ext}} \psi(\xi)(1 - \psi(\xi)), && \quad \xi \in (0, T], \\[3mm]
\psi'(\xi) & = \chi(\xi), && \quad \xi \in (0, T],
\end{aligned}
\end{equation}
assuming an isotropic diffusion tensor $\mathbf{D} = d_{ext} \mathbf{I}$. For this test case, we use the parameters $d_{ext} = 10^{-3}$, $\alpha = 1$, and $\eta_0 = 10$. The wave speed is set to $v = 0.1$, with initial conditions $\psi(0) = 1$ and $\chi(0) = -10^{-2}$. The chosen parameters are taken from \cite{corti_discontinuous_2023} to allow a direct comparison between the proposed model and the reference. We consider two final simulation times, $T = 5$ and $T = 10$ and $\Omega =(0,5)\times(0,1)$. We compute a reference solution of problem \eqref{eq:wave_problem} numerically using MATLAB’s \texttt{ode45} solver, with a tolerance set to $10^{-12}$.  
\par
First, we analyse the effect of different mesh sizes ($N_\mathrm{el} = 30, 100, 300$), with the corresponding number of degrees of freedom (DOFs) keeping the time step fixed at $\Delta t = 0.01$, and compute the $L^2$-error at the final simulation time for both $p=2$ and $p=3$. The results are reported in Table~\ref{tab:combined_errors} for both the FK and SFK methods. We observe that the stabilised version significantly reduces the errors, particularly for coarse meshes with a lower number of degrees of freedom. Remarkably, our model is able to simulate the travelling wave even for $p=2$. Figure~\ref{fig:testcase2_waveplot} illustrates the final-time solutions for both methods, providing a visual counterpart to the numerical data in Table~\ref{tab:combined_errors}. For $T=10$, the solutions produced by the original model are completely unreliable, whereas our formulation provides reasonably accurate results in all cases. The numerical scheme proposed in \cite{corti_discontinuous_2023} lacks robustness when the solution approaches negative values, due to the instability of the equilibrium point at $c = 0$. 
\par
Figure~\ref{fig:positivity_vs_nonpositivity} illustrates an example where the error-correction mechanism is active. We report the solution along the line $y=0.5$ for $x \in [0,5]$, with $p=2$ and $N_\mathrm{el}=30$, using both schemes. For the stabilised scheme, when the solution reaches negative values, the model corrects the oscillations after a few iterations. In contrast, the scheme without positivity preservation amplifies the oscillations around $c=0$, leading to the results shown in Figure~\ref{fig:testcase2_waveplot}. Additionally, with the classical formulation, we observe oscillations also around 
$c=1$, which are typical in applications with fast variations (high gradients), 
in addition to those occurring at $c=0$. As shown, our method is able to damp 
the oscillations at $c=1$ as well, leading to overall improved results.
\begin{figure}[ht]
    \centering
    \includegraphics[width=\linewidth]{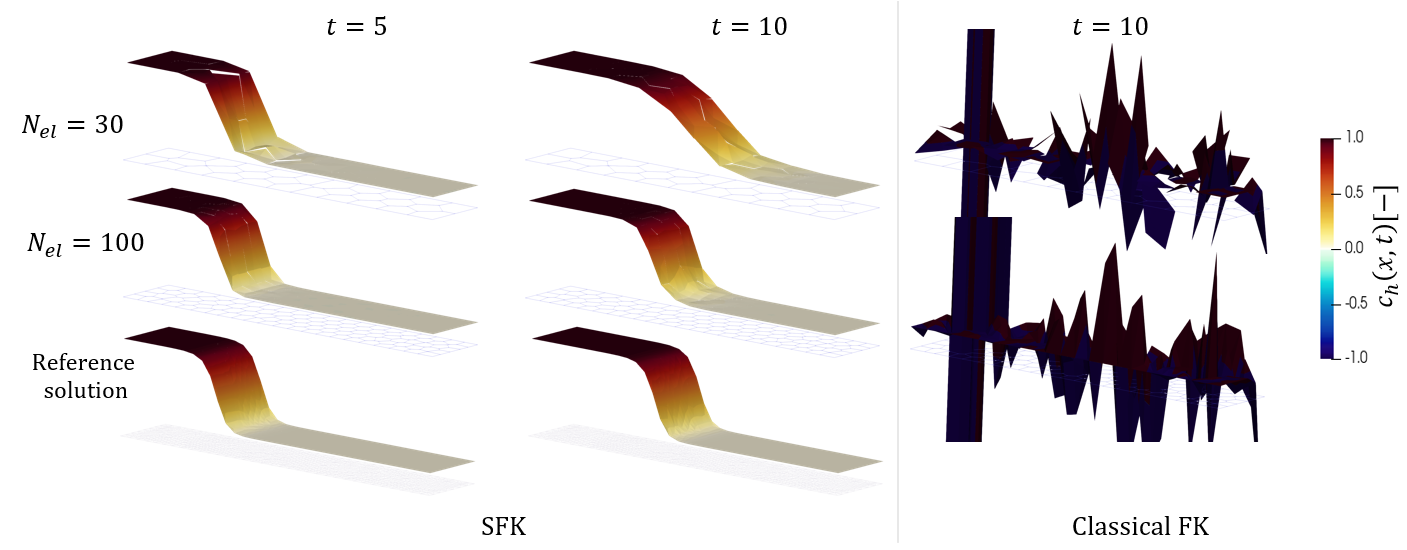}
    \caption{Test case 3: solutions with different mesh refinements using both the classical and stabilised FK formulations with $p=2$.}
    \label{fig:testcase2_waveplot}
\end{figure}

\begin{table}[!t]
\centering
\caption{Test case 3: computed $L^2$ errors at final time with different mesh refinements ($\Delta t = 0.01$): $p = 2$ (top) and $p = 3$ (bottom) for both the classical and stabilised FK formulations.}
\label{tab:combined_errors}

\small 
\setlength{\tabcolsep}{8pt} 
\renewcommand{\arraystretch}{1.3} 

\begin{tabular}{l|l|cc|cccc}
\hline
& & \multicolumn{2}{c}{\textbf{SFK}} & \multicolumn{4}{c}{\textbf{FK} \cite{corti_discontinuous_2023}} \\
\cline{3-8}
& & \multicolumn{2}{c|}{Semi-Implicit} & \multicolumn{2}{c}{Semi-Implicit} & \multicolumn{2}{c}{Implicit} \\
\cline{3-8}
Refinement & DOFs & $T = 5$ & $T = 10$ & $T = 5$ & $T = 10$ & $T = 5$ & $T = 10$ \\
\hline
\hline
\multicolumn{8}{c}{$p = 2 \quad \Delta t = 0.01$} \\
\hline
$N_{el} = 30$  & 180  & $7.00 \times 10^{-2}$ & $4.74 \times 10^{-1}$ & $6.33 \times 10^2$  & $1.03 \times 10^4$ & $1.63 \times 10^9$   & $5.36 \times 10^6$ \\
$N_{el} = 100$ & 600  & $5.55 \times 10^{-3}$ & $8.72 \times 10^{-2}$  & $1.45 \times 10^2$  & $1.12 \times 10^4$ & $8.24 \times 10^{-1}$ & $2.18 \times 10^7$ \\
$N_{el} = 300$ & 1800 & $3.00 \times 10^{-3}$ & $3.70 \times 10^{-3}$  & $1.98 \times 10^1$  & $6.02 \times 10^4$ & $6.97 \times 10^{-2}$ & $1.28 \times 10^8$ \\
\hline
\hline
\multicolumn{8}{c}{$p = 3 \quad \Delta t = 0.01$} \\
\hline
$N_{el} = 30$  & 300  & $1.41 \times 10^{-2}$ & $2.81 \times 10^{-1}$ & $9.27 \times 10^{-1}$ & $6.75 \times 10^4$  & $1.56 \times 10^{-1}$ & $7.34 \times 10^7$ \\
$N_{el} = 100$ & 1000 & $2.70 \times 10^{-3}$ & $6.20 \times 10^{-3}$  & $5.50 \times 10^{-2}$ & $7.20 \times 10^{-1}$ & $8.12 \times 10^{-3}$ & $1.78 \times 10^{-1}$ \\
$N_{el} = 300$ & 3000 & $2.70 \times 10^{-3}$ & $5.90 \times 10^{-3}$  & $6.35 \times 10^{-4}$ & $7.80 \times 10^{-3}$ & $7.71 \times 10^{-4}$ & $2.12 \times 10^{-3}$ \\
\hline
\end{tabular}
\end{table}
In Figure~\ref{fig:testcase2_err_vs_T_vs_DOFs}, the left panel reports the error in the energy functional defined in \eqref{eq:energyNorm} for three polynomial orders as a function of time. The behaviour for \(p=2\) is qualitatively consistent with the other cases, further supporting the improved long-time stability predicted by Theorem~\ref{thm:stability_estimate}. In particular, no exponential growth is observed for low-order approximations, in contrast with the formulation in \cite{corti_discontinuous_2023}. The right panel shows the error at the final time \(T=5\) as a function of the number of DOFs, comparing \(h\)-refinement with fixed polynomial order \(p=2\) and \(p\)-refinement on a fixed mesh with \(N_\mathrm{el}=30\). For a fixed number of DOFs, higher polynomial orders consistently yield lower errors than mesh refinement, indicating that \(p\)-refinement is more effective than \(h\)-refinement in this setting.

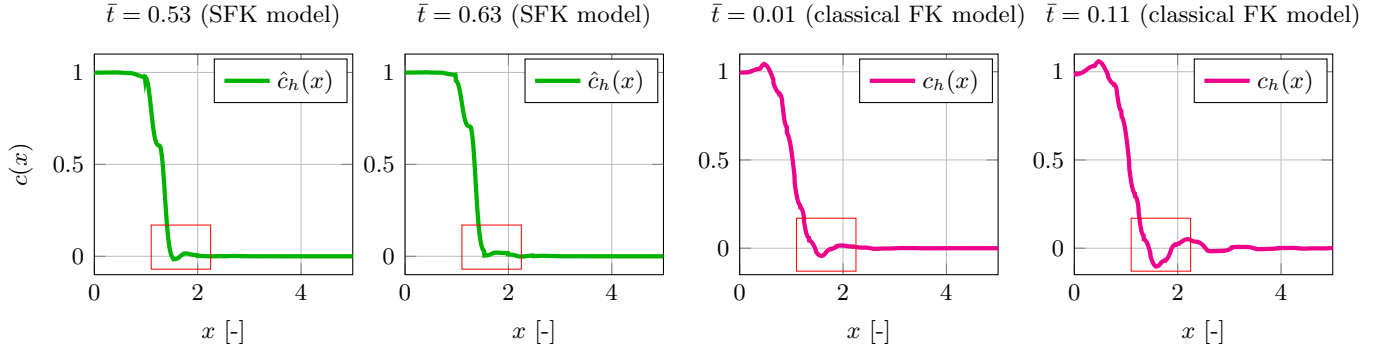
\begin{figure}[!t] 
  \centering

  \begin{subfigure}{0.245\textwidth} 
    \centering
    \begin{tikzpicture}
      \begin{axis}[
        width=5.0cm,
        height=4.5cm, 
        font=\footnotesize, 
        xlabel={$x$ [-]},
        ylabel={$c(x)$},
        grid=both,
        legend style={at={(0.98,0.98)},anchor=north east},
        ymin=-0.1,
        ymax=1.1,
        title={$\bar t=0.53$ (SFK model)},
        xmin = 0,
        xmax=5
      ]
        \addplot[
          color=green!70!black,
          ultra thick
        ] table[
          col sep=comma,
          x=x,
          y=uh
        ] {Dati/Snapshot53_positivity_preserving.csv};
        \addlegendentry{$\hat{c}_h(x)$}
        \draw[red, ultra thin] (1.1,-0.07) rectangle (2.25,0.17);
      \end{axis}
    \end{tikzpicture}
  \end{subfigure}
  \hfill
  \begin{subfigure}{0.245\textwidth} 
    \centering
    \begin{tikzpicture}
      \begin{axis}[
        width=5.0cm,
        height=4.5cm, 
        font=\footnotesize, 
        xlabel={$x$ [-]},
        grid=both,
        ymin=-0.1,
        ymax=1.1,
        xmin = 0,
        xmax=5,
        title={$\bar t=0.63$ (SFK model)},
        legend style={at={(0.98,0.98)},anchor=north east},
      ]
        \addplot[
          color=green!70!black,
          ultra thick,
        ] table[
          col sep=comma,
          x=x,
          y=uh
        ] {Dati/Snapshot63_positivity_preserving.csv};
        \addlegendentry{$\hat{c}_h(x)$}
        \draw[red, ultra thin] (1.1,-0.07) rectangle (2.25,0.17);
      \end{axis}
    \end{tikzpicture}
  \end{subfigure}
  \begin{subfigure}{0.245\textwidth} 
    \centering
    \begin{tikzpicture}
      \begin{axis}[
        width=5.0cm,
        height=4.5cm, 
        font=\footnotesize, 
        xlabel={$x$ [-]},
        grid=both,
        legend style={at={(0.98,0.98)},anchor=north east},
        ymin=-0.15,
        ymax=1.1,
        title={$\bar t=0.01$ (classical FK model)},
        xmin = 0,
        xmax=5
      ]
        \addplot[
          color=magenta,
          ultra thick
        ] table[
          col sep=comma,
          x=x,
          y=uh
        ] {Dati/Snapshot_FK_0.csv};
        \addlegendentry{$c_h(x)$}
        \draw[red, ultra thin] (1.1,-0.13) rectangle (2.25,0.17);
      \end{axis}
    \end{tikzpicture}
  \end{subfigure}
  \hfill
  \begin{subfigure}{0.245\textwidth} 
    \centering
    \begin{tikzpicture}
      \begin{axis}[
        width=5.0cm,
        height=4.5cm, 
        font=\footnotesize, 
        xlabel={$x$ [-]},
        grid=both,
        ymin=-0.15,
        ymax=1.1,
        xmin = 0,
        xmax=5,
        title={$\bar{t} =0.11$ (classical FK model)},
        legend style={at={(0.98,0.98)},anchor=north east},
      ]
        \addplot[
          color=magenta,
          ultra thick,
        ] table[
          col sep=comma,
          x=x,
          y=uh
        ] {Dati/Snapshot_FK_10.csv};
        \addlegendentry{$c_h(x)$}
        \draw[red, ultra thin] (1.1,-0.13) rectangle (2.25,0.17);
      \end{axis}
    \end{tikzpicture}
  \end{subfigure}

  \vspace{0.1cm}
\caption{Test case 3: snapshots of the approximated solution along the line $y=0.5$ at different timestamps for the SFK model and the classical FK model \cite{corti_discontinuous_2023}.}
  \label{fig:positivity_vs_nonpositivity}
\end{figure}

\begin{figure}[t]
\centering

\begin{subfigure}[c]{0.46\textwidth} 
\centering
\begin{tikzpicture}
\begin{semilogyaxis}[
    width=\textwidth,
    height=4.5cm,
    font=\footnotesize, 
    xlabel={$T$ [-]},
    ylabel={$\|c(t) - c_h(t)\|_{\varepsilon} $},
    grid=both,
    grid style={line width=.1pt, draw=gray!30},
    xtick={1,...,8},
    xmin=0,
    xmax=5,
    thick,
    cycle list={
        {green!60!black, solid, line width=1.5pt},
        {orange!80!black, solid, line width=1.5pt},
        {cyan!80!black, solid, line width=1.5pt}
    },
    legend style={at={(0.02,0.98)},anchor=north west}
]

\addplot table [x=T, y=err_Energy, col sep=comma] {Dati/T_vs_Energy_p_2.csv}; \addlegendentry{$p=2$}
\addplot table [x=T, y=err_Energy, col sep=comma] {Dati/T_vs_Energy_p_3.csv}; \addlegendentry{$p=3$}
\addplot table [x=T, y=err_Energy, col sep=comma] {Dati/T_vs_Energy_p_4.csv}; \addlegendentry{$p=4$}

\end{semilogyaxis}
\end{tikzpicture}
\end{subfigure}
\hfill 
\begin{subfigure}[c]{0.46\textwidth} 
\centering
\begin{tikzpicture}
\begin{loglogaxis}[
    width=\textwidth,
    height=4.5cm,
    font=\footnotesize, 
    xlabel={DOFs },
    ylabel={$\|c(t) - c_h(t)\|_{\varepsilon} $},
    grid=both,
    grid style={line width=.1pt, draw=gray!30},
    thick,
    xticklabel={\pgfmathparse{exp(\tick)}\pgfmathprintnumber{\pgfmathresult}},
x tick label style={
/pgf/number format/.cd, fixed, fixed zerofill,
precision=0},
    ymin=0.001,
    cycle list={
        {green!60!black, solid, line width=1.5pt},
        {orange!80!black, solid, line width=1.5pt},
    },
    legend style={at={(0.98,0.98)},anchor=north east}
]

\addplot table [x=DOFs, y=EnergyError, col sep=comma] {Dati/p-ref.csv}; \addlegendentry{$p$-refinement}
\addplot table [x=DOFs, y=EnergyError, col sep=comma] {Dati/h-ref_p_2.csv}; \addlegendentry{$h$-refinement ($p=2$)}

\end{loglogaxis}
\end{tikzpicture}
\end{subfigure}

\caption{Test case 3: computed errors with respect to the simulation time $T$ for $N_{el} = 30$ (left) and with respect to the number of degrees of freedom (DOFs) considering $p$-refinement and $h$-refinement (right).}
\label{fig:testcase2_err_vs_T_vs_DOFs}
\end{figure}
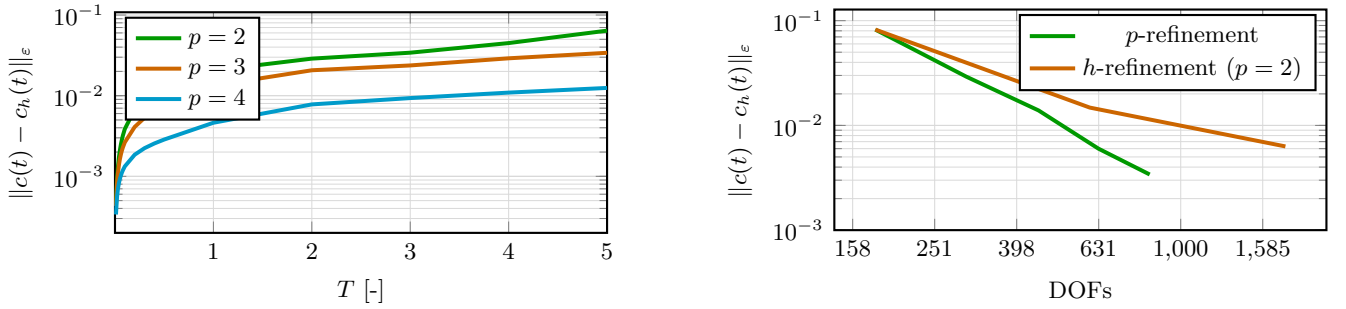
\FloatBarrier

\subsection{Test Case 4: Spreading of \texorpdfstring{$\alpha$}{alpha}-synuclein in a 2D brain section}
\begin{figure}[!t]
    \centering
    \includegraphics[width=0.8\textwidth]{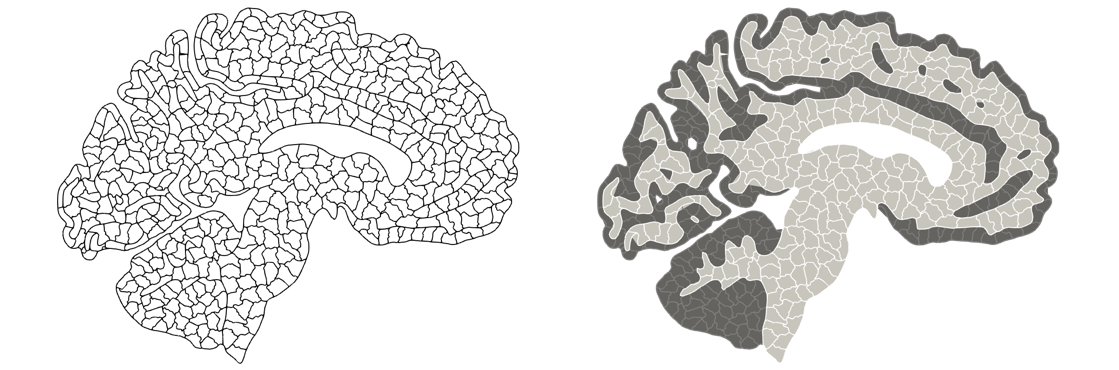}
    \caption{Sagittal section of the brain domain showing the agglomerated polytopal mesh with $534$ elements (left) and the corresponding computational domain partitioned into grey and white matter (right).}
    \label{fig:brain_meshes}
\end{figure}
In this section, we simulate the spreading of $\alpha$-synuclein protein on a polytopal agglomerated grid representing a brain section. We use an agglomerated mesh made of 534 elements that represents a sagittal brain slice as shown in Figure~\ref{fig:brain_meshes}. The solution is computed by using a polynomial order of discretisation $p=4$ and a time-step $\Delta t = 10^{-2} $ years. Concerning the parameters of the model, the diffusion tensor $\mathbf{D} = \mathbf{D}(\boldsymbol{x})$ describes the spreading of misfolded proteins inside the domain and is obtained as the combination of an axonal diffusion of magnitude $d_{axn} \geq 0$ and an extracellular diffusion effect of magnitude $d_{ext} \geq 0$, such that:
\begin{equation}
    \mathbf{D}(\boldsymbol{x}) = d_{ext} \mathbf{I} + d_{axn} \boldsymbol{a}(\boldsymbol{x}) \otimes \boldsymbol{a}(\boldsymbol{x}),
\end{equation}
where $\boldsymbol{a}(\boldsymbol{x})$ is the axonal fibres orientation at point $\boldsymbol{x} \in \Omega$. The axonal direction is the main orientation of the connections between the neurons (axons), which can be derived from Diffusion Weighted Imaging (DWI) \cite{weickenmeier_physics_2019}.

We select an extracellular diffusion magnitude $d_{ext} = 8  \ \text{mm}^2 / \text{year}$ and an axonal diffusion magnitude $d_{axn} = 80 \ \text{mm}^2 / \text{year}$ \cite{schafer_interplay_2019} (see \cite[Fig. 6b]{antonietti_discontinuous_2024} for details axonal fibres). Moreover, we set the reaction parameter $\alpha = 0.9 / \text{year}$ and the penalty coefficient $\eta_0 = 10$. We assume $f = 0$ and impose homogeneous Neumann boundary conditions on $\partial \Omega$. As an initial condition, we consider a Gaussian concentration of proteins initially located in the dorsal motor nucleus \cite{braak_staging_2003}. 
\begin{figure}[!t]
    \centering
    \includegraphics[width=\textwidth]{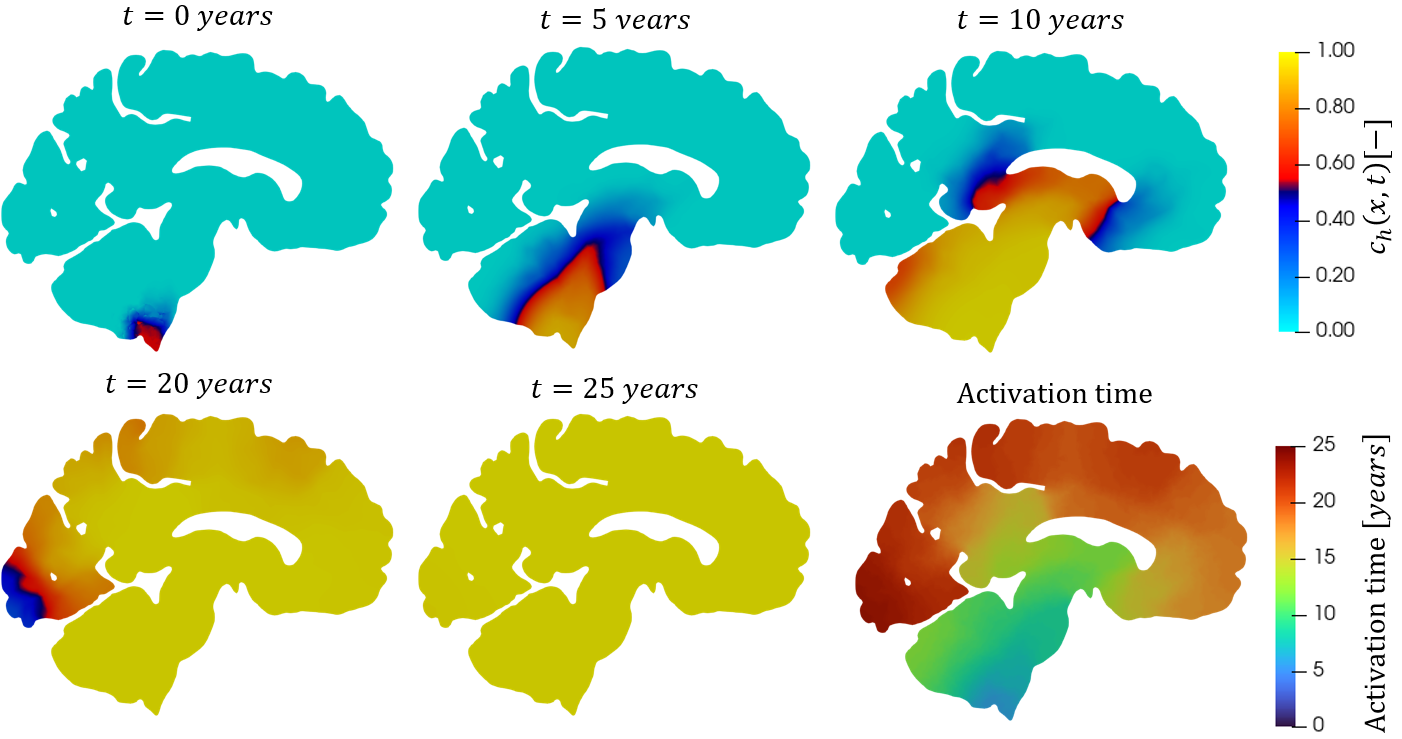}

    \caption{Test Case 4: patterns of computed $\alpha$-synuclein concentration $\hat{c}_h$ through different years of development of the pathology and corresponding activation time computed with the SFK model and polynomial order $p=4$.}
    \label{fig:mfk_patterns_and_activation}
\end{figure}
\par
In Figure~\ref{fig:mfk_patterns_and_activation}, we report the initial condition and the solution at different time instants up to the final time $T = 25 $ years. We observe that the diffusion of the misfolded proteins is coherent with the medical literature \cite{braak_staging_2003, goedert_alzheimer_2015}. We highlight that the time development of the pathology is of the order of 25 years, consistent with the medical literature. Furthermore, we compute the activation time of the pathology as:
\begin{equation}
    \hat{t}(\boldsymbol{x}) = \underset{t\in(0,T]}{\mathrm{argmin}} \{\hat{c}_h(\boldsymbol{x},t)\geq c_\mathrm{crit}\}
\label{activationtime}
\end{equation}
where we set the critical concentration of misfolded proteins $c_\mathrm{crit} =0.95$. The choice of this value is justified by the fact that higher values of protein concentration compromise the electric signal transport. The activation time gives us the idea of the time after which a neuron is damaged due to the progression of the disease. 
\subsection{Test Case 5: Comparison with the FK model}
We repeat the previous test case using the fully discrete FK formulation introduced in \cite{corti_discontinuous_2023}. Keeping the same mesh and parameter setup, we simulate the spreading of $\alpha$-synuclein in a 2D brain section. Figure~\ref{fig:fk_pattern_activation} compares the numerical solutions of the FK and SFK models, showing excellent consistency in the absence of negative values. To further assess the formulation, we compare the activation times computed by both models. The resulting activation profiles are identical due to the exact agreement of the solutions at each time point, further validating the consistency between the FK and SFK frameworks. Moreover, increasing the polynomial order and refining the mesh effectively suppresses numerical oscillations, confirming that both models converge toward the same propagation pattern under discretisation refinement.
\begin{figure}
    \centering\includegraphics[width=\textwidth]{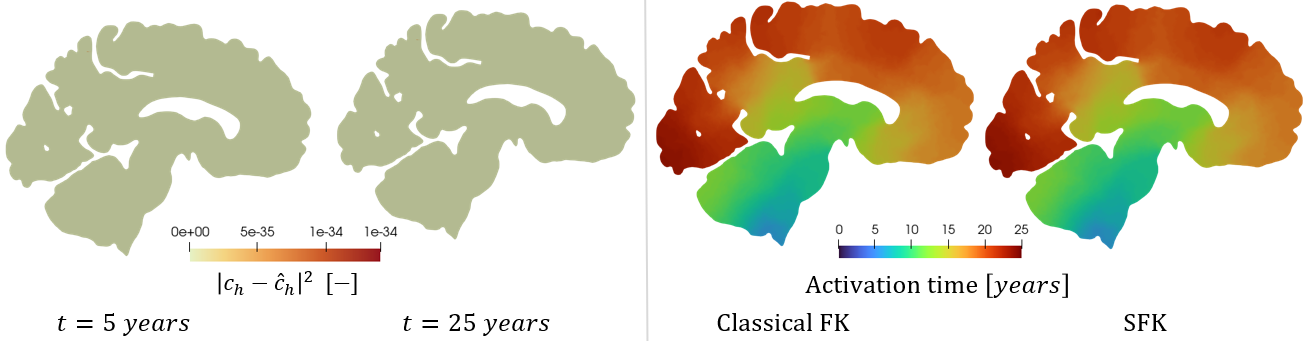}
    \caption{Test Case 5: comparison of the FK and SFK ($p=4$) models. Left panels depict the spatial difference in $\alpha$-synuclein concentration over time. The right panel shows the corresponding activation times for both models.}
    \label{fig:fk_pattern_activation}
\end{figure}
\begin{figure}
    \centering\includegraphics[width=\linewidth]{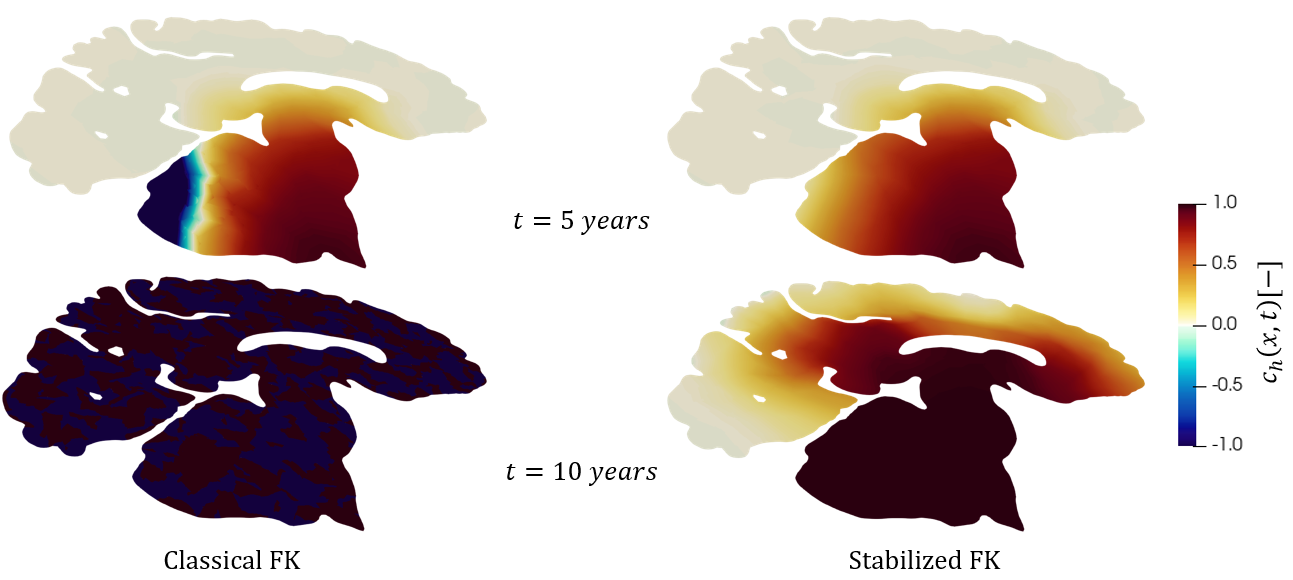}
    \caption{Test Case 5: comparison of computed $\alpha$-synuclein concentration across different years of pathology development for the FK and SFK models with $p=1$.}
    \label{fig:fk_vs_mfk_p1_brain}
\end{figure}

To assess the robustness of the SFK model under lower-order discretisations, we repeat the simulation using $p = 1$ for both formulations. As shown in Figure~\ref{fig:fk_vs_mfk_p1_brain}, the SFK model accurately captures the expected propagation dynamics. In contrast, the FK model suffers from numerical instability, producing unphysical negative values near the outer brain boundary that ultimately trigger divergence. This comparison underscores the ability of the SFK formulation to deliver reliable solutions even at low polynomial orders, significantly reducing computational overhead.
\subsection{Test Case 6: Spreading of \texorpdfstring{$\alpha$}{alpha}-synuclein with grey-white matter distinction}
In this test case, we simulate the spreading of $\alpha$-synuclein protein in a 2D brain section taking into account the distinction between grey and white matter (see Figure~\ref{fig:brain_meshes}). We highlight the difference because grey matter can be approximately assumed isotropic, while white matter is mainly made of axons, whose fibres represent the principal direction of propagation of the electric signal. We select a unique extracellular diffusion magnitude $d_{ext} = 8 \ \text{mm}^2 / \text{year}$, an axonal diffusion magnitude $d_{axn} = 0  \ \text{mm}^2 / \text{year}$ and a reaction parameter $\alpha = 0.45 / \text{year}$ for grey matter regions, and $d_{axn} = 80 \ \text{mm}^2 / \text{year}$ and $\alpha = 0.9 / \text{year}$ for white matter regions. The solution is computed using a time-step $\Delta t = 10^{-2}$ years and polynomial order $p=4$ for both formulations, on the agglomerated mesh introduced in Test Case 3. Moreover, we choose $\eta_0 = 10$, $f = 0$, and we assume homogeneous Neumann boundary conditions on $\partial \Omega$. As an initial condition, we consider a Gaussian distribution of proteins located in the dorsal motor nucleus \cite{braak_staging_2003}.
\par
Figure~\ref{fig:fk_vs_mfk_gb} displays the computed solution at different time instants ($t = 10$ and $T = 25$ years) alongside the resulting activation times for both formulations. We observe that the progression of the disease remains coherent with the medical literature \cite{braak_staging_2003}. The activation time of the pathology, computed using \eqref{activationtime}, is not complete within the considered time frame, as expected from the concentration patterns. Moreover, the SFK model achieves the same accuracy and captures identical patterns as the FK model at the same polynomial order ($p=4$), confirming the consistency between the two formulations also in this more anatomically detailed setting, where grey and white matter regions exhibit markedly different diffusion and reaction properties.
\begin{figure}[t]
    \centering
    \includegraphics[width=\textwidth]{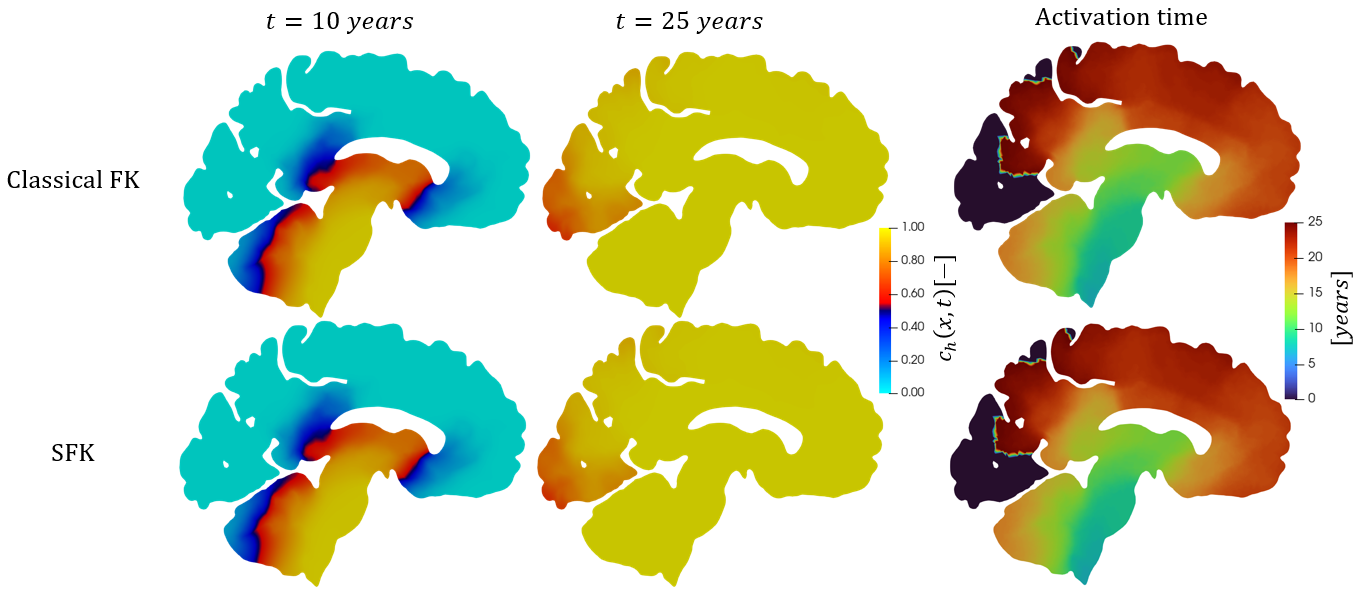}
    \caption{Test Case 6: patterns of computed $\alpha$-synuclein concentration at $t=10$ and $T=25$ years, and corresponding pathology activation times, distinguishing between grey and white matter. Results are obtained using the FK model ($p=4$) and the SFK model ($p=4$).}
    \label{fig:fk_vs_mfk_gb}
\end{figure}
\section{Conclusions}
\label{sec:end}
In this work, we proposed a polytopal discontinuous Galerkin discretisation of a stabilised Fisher--Kolmogorov model obtained by introducing an absolute-value stabilisation in the non-linear reaction term. For its numerical approximation, we propose a discontinuous Galerkin spatial discretisation on general polygonal and polyhedral meshes with the Crank--Nicolson time integration scheme. The resulting model and numerical scheme exhibit several key advantages. On the one hand, for non-negative continuous solutions, the formulation is strongly consistent with the original model, which is particularly relevant for biological applications. On the other hand, it ensures significantly improved stability properties. 
We prove stability and convergence error estimates for the proposed method on a general class of polytopal meshes. 
We provided numerical validation of the theoretical results in both two and three dimensions. We also simulated the diffusion of $\alpha$-synuclein in a realistic 2D brain section, validating the outcomes by comparing activation times in selected brain regions with findings reported in the medical literature. Furthermore, a comparison with the non-stabilised formulation shows that the proposed method delivers reliable results even at lower computational cost, while maintaining high-order accuracy and computational efficiency.
\par
Possible future developments include validation on realistic 3D brain geometries, enabling applications to patient-specific reconstructions. Another promising direction is the development of a space--time DG formulation \cite{mazzieri_space-time_2020,antonietti_discontinuous_2023} to achieve higher-order convergence in time. Finally, the approach could be extended to other biological population dynamics models affected by similar instabilities near equilibrium states, which often compromise the robustness of standard numerical solvers.


\section*{Acknowledgements}
OASIS-3 provided the brain MRI images: Longitudinal Multimodal Neuroimaging: Principal Investigators: T. Benzinger, D. Marcus, J. Morris; NIH P30 AG066444, P50 AG00561, P30 NS09857781, P01 AG026276, P01 AG003991, R01 AG043434, UL1 TR000448, R01 EB009352. AV-45 doses were provided by Avid Radiopharmaceuticals, a wholly-owned subsidiary of Eli Lilly.

\bibliographystyle{ieeetr}
\bibliography{sample.bib}
\end{document}

%% file: Images/h_2Dconv_L2.tex
%
%
\definecolor{mycolor1}{rgb}{0.00000,0.44700,0.74100}%
\definecolor{mycolor2}{rgb}{0.85000,0.32500,0.09800}%
\definecolor{mycolor3}{rgb}{0.92900,0.69400,0.12500}%
\definecolor{mycolor4}{rgb}{0.49400,0.18400,0.55600}%
\definecolor{mycolor5}{rgb}{0.46600,0.67400,0.18800}%
\definecolor{mycolor6}{rgb}{0.30100,0.74500,0.93300}%
\begin{tikzpicture}

\begin{axis}[%
width=4.0in,
height=2.5in,
scale only axis,
xmode=log,
xmin=0.0594220932193747,
xmax=0.330150595398318,
xminorticks=false,
xlabel style={font=\color{white!15!black}},
xlabel={$h[-]$},
xticklabel={\pgfmathparse{exp(\tick)}\pgfmathprintnumber{\pgfmathresult}},
x tick label style={
/pgf/number format/.cd, fixed, fixed zerofill,
precision=2},
ymode=log,
ymin=1e-12,
ymax=0.01,
yminorticks=false,
ylabel style={font=\color{white!15!black}},
ylabel={$\|\hat{c}(T) - \hat{c}_h(T)\|$},
axis background/.style={fill=white},
xmajorgrids,
ymajorgrids,
legend style={legend cell align=left, align=left, draw=white!15!black, at={(0.06,0.96)}, anchor=north west}
]

\addplot [color=mycolor1, line width=2.0pt, mark=*]
  table[row sep=crcr]{%
0.330150595398318	0.0017032053351897\\
0.183622722177214	0.000532715700730886\\
0.107072661625942	0.000205101548275546\\
0.0594220932193747	7.29478311671391e-05\\
};
\addlegendentry{$p=1$}

\addplot [color=mycolor2, line width=2.0pt, mark=*]
  table[row sep=crcr]{%
0.330150595398318	0.00039298769862688\\
0.183622722177214	7.44306671304848e-05\\
0.107072661625942	1.30773456250953e-05\\
0.0594220932193747	2.2093005662909e-06\\
};
\addlegendentry{$p=2$}

\addplot [color=mycolor3, line width=2.0pt, mark=*]
  table[row sep=crcr]{%
0.330150595398318	4.73215479264294e-05\\
0.183622722177214	3.10360196402689e-06\\
0.107072661625942	3.15840092568419e-07\\
0.0594220932193747	2.57723474023662e-08\\
};
\addlegendentry{$p=3$}

\addplot [color=mycolor4, line width=2.0pt, mark=*]
  table[row sep=crcr]{%
0.330150595398318	2.53050954562652e-06\\
0.183622722177214	1.35825526286931e-07\\
0.107072661625942	7.81371143870854e-09\\
0.0594220932193747	3.75042256236892e-10\\
};
\addlegendentry{$p=4$}

\addplot [color=mycolor5, line width=2.0pt, mark=*]
  table[row sep=crcr]{%
0.330150595398318	1.07611260356472e-07\\
0.183622722177214	2.46290421711425e-09\\
0.107072661625942	8.18538790095752e-11\\
0.0594220932193747	1.08797635705023e-11\\
};
\addlegendentry{$p=5$}

\addplot [color=mycolor6, line width=2.0pt, mark=*]
  table[row sep=crcr]{%
0.330150595398318	3.9513129501564e-09\\
0.183622722177214	5.6176023210469e-11\\
0.107072661625942	1.07376273178609e-11\\
0.0594220932193747	1.06870511798269e-11\\
};
\addlegendentry{$p=6$}

\draw[black, thick]
(axis cs:0.22, 5.5e-4) -- 
(axis cs:0.25, 5.5e-4) -- 
(axis cs:0.25, 7.875e-4) -- cycle; 
\node[right, xshift=2pt] at (axis cs:0.25, 6.9e-4) {2};

\draw[black, thick]
(axis cs:0.22, 9.0e-5) -- 
(axis cs:0.25, 9.0e-5) -- 
(axis cs:0.25, 14.25e-5) -- cycle; 
\node[right, xshift=2pt] at (axis cs:0.25, 9.5e-5) {3};

\draw[black, thick]
(axis cs:0.22, 4.0e-6) -- 
(axis cs:0.25, 4.0e-6) -- 
(axis cs:0.25, 8.906e-6) -- cycle; 
\node[right, xshift=2pt] at (axis cs:0.25, 6.0e-6) {4};

\draw[black, thick]
(axis cs:0.22, 2.0e-7) -- 
(axis cs:0.25, 2.0e-7) -- 
(axis cs:0.25, 4.441e-7) -- cycle; 
\node[right, xshift=2pt] at (axis cs:0.25, 3.0e-7) {5};

\draw[black, thick]
(axis cs:0.22, 5.0e-9) -- 
(axis cs:0.25, 5.0e-9) -- 
(axis cs:0.25, 13.441e-9) -- cycle; 
\node[right, xshift=2pt] at (axis cs:0.25, 8.0e-9) {6};

\draw[black, thick]
(axis cs:0.22, 1.5e-10) -- 
(axis cs:0.25, 1.5e-10) -- 
(axis cs:0.25, 4.2441e-10) -- cycle; 
\node[right, xshift=2pt] at (axis cs:0.25, 2.5e-10) {7};

\end{axis}
\end{tikzpicture}%

%% file: Images/h_2Dconv_DG.tex
%
%
\definecolor{mycolor1}{rgb}{0.00000,0.44700,0.74100}%
\definecolor{mycolor2}{rgb}{0.85000,0.32500,0.09800}%
\definecolor{mycolor3}{rgb}{0.92900,0.69400,0.12500}%
\definecolor{mycolor4}{rgb}{0.49400,0.18400,0.55600}%
\definecolor{mycolor5}{rgb}{0.46600,0.67400,0.18800}%
\definecolor{mycolor6}{rgb}{0.30100,0.74500,0.93300}%
\begin{tikzpicture}

\begin{axis}[%
width=4.0in,
height=2.5in,
scale only axis,
xmode=log,
xmin=0.0594220932193747,
xmax=0.330150595398318,
xminorticks=false,
xlabel style={font=\color{white!15!black}},
xlabel={$h[-]$},
ymode=log,
ymin=1e-12,
ymax=1,
yminorticks=false,
xticklabel={\pgfmathparse{exp(\tick)}\pgfmathprintnumber{\pgfmathresult}},
x tick label style={
/pgf/number format/.cd, fixed, fixed zerofill,
precision=2},
ylabel style={font=\color{white!15!black}},
ylabel={$\|\hat{c}(T) - \hat{c}_h(T)\|_\mathrm{DG}$},
axis background/.style={fill=white},
xmajorgrids,
ymajorgrids,
legend style={legend cell align=left, align=left, draw=white!15!black, at={(0.06,0.96)}, anchor=north west}
]
\addplot [color=mycolor1, line width=2.0pt, mark=*]
  table[row sep=crcr]{%
0.330150595398318	0.0954642030680627\\
0.183622722177214	0.0388217526158685\\
0.107072661625942	0.0166633748454037\\
0.0594220932193747	0.00774749805880071\\
};
\addlegendentry{$p=1$}

\addplot [color=mycolor2, line width=2.0pt, mark=*]
  table[row sep=crcr]{%
0.330150595398318	0.0429392425668809\\
0.183622722177214	0.0133946134005278\\
0.107072661625942	0.00440471829786157\\
0.0594220932193747	0.00134547426743042\\
};
\addlegendentry{$p=2$}

\addplot [color=mycolor3, line width=2.0pt, mark=*]
  table[row sep=crcr]{%
0.330150595398318	0.00567197775954265\\
0.183622722177214	0.000736087835384694\\
0.107072661625942	0.000135399957239786\\
0.0594220932193747	2.10779820988678e-05\\
};
\addlegendentry{$p=3$}

\addplot [color=mycolor4, line width=2.0pt, mark=*]
  table[row sep=crcr]{%
0.330150595398318	0.000504035704428251\\
0.183622722177214	4.14235781585214e-05\\
0.107072661625942	4.33725725062076e-06\\
0.0594220932193747	3.57639274311806e-07\\
};
\addlegendentry{$p=4$}

\addplot [color=mycolor5, line width=2.0pt, mark=*]
  table[row sep=crcr]{%
0.330150595398318	2.798315815265e-05\\
0.183622722177214	1.02645308633114e-06\\
0.107072661625942	6.21886617994587e-08\\
0.0594220932193747	2.60468945941712e-09\\
};
\addlegendentry{$p=5$}

\addplot [color=mycolor6, line width=2.0pt, mark=*]
  table[row sep=crcr]{%
0.330150595398318	1.46948989402931e-06\\
0.183622722177214	3.58684085213035e-08\\
0.107072661625942	1.08677547777119e-09\\
0.0594220932193747	8.22378373629531e-11\\
};
\addlegendentry{$p=6$}

\draw[black, thick]
(axis cs:0.22, 3.5e-2) -- 
(axis cs:0.25, 3.5e-2) -- 
(axis cs:0.25, 4.875e-2) -- cycle; 
\node[above, xshift=2pt] at (axis cs:0.265, 2.5e-2) {1};

\draw[black, thick]
(axis cs:0.22, 1.0e-2) -- 
(axis cs:0.25, 1.0e-2) -- 
(axis cs:0.25, 1.5e-2) -- cycle; 
\node[right, xshift=2pt] at (axis cs:0.25, 1.25e-2) {2};

\draw[black, thick]
(axis cs:0.22, 8.0e-4) -- 
(axis cs:0.25, 8.0e-4) -- 
(axis cs:0.25, 14.906e-4) -- cycle; 
\node[right, xshift=2pt] at (axis cs:0.25, 1.2e-3) {3};

\draw[black, thick]
(axis cs:0.22, 4.5e-5) -- 
(axis cs:0.25, 4.5e-5) -- 
(axis cs:0.25, 10.906e-5) -- cycle; 
\node[right, xshift=2pt] at (axis cs:0.25, 7.0e-5) {4};

\draw[black, thick]
(axis cs:0.22, 1.5e-6) -- 
(axis cs:0.25, 1.5e-6) -- 
(axis cs:0.25, 3.841e-6) -- cycle; 
\node[right, xshift=2pt] at (axis cs:0.25, 2.5e-6) {5};

\draw[black, thick]
(axis cs:0.22, 6.0e-8) -- 
(axis cs:0.25, 6.0e-8) -- 
(axis cs:0.25, 18.441e-8) -- cycle; 
\node[right, xshift=2pt] at (axis cs:0.25, 9.5e-8) {6};

\end{axis}
\end{tikzpicture}%

%% file: Images/h_2Dconv_eps.tex
%
%
\definecolor{mycolor1}{rgb}{0.00000,0.44700,0.74100}%
\definecolor{mycolor2}{rgb}{0.85000,0.32500,0.09800}%
\definecolor{mycolor3}{rgb}{0.92900,0.69400,0.12500}%
\definecolor{mycolor4}{rgb}{0.49400,0.18400,0.55600}%
\definecolor{mycolor5}{rgb}{0.46600,0.67400,0.18800}%
\definecolor{mycolor6}{rgb}{0.30100,0.74500,0.93300}%
\begin{tikzpicture}

\begin{axis}[%
width=4.0in,
height=2.5in,
scale only axis,
xmode=log,
xmin=0.0594220932193747,
xmax=0.330150595398318,
xminorticks=false,
xticklabel={\pgfmathparse{exp(\tick)}\pgfmathprintnumber{\pgfmathresult}},
x tick label style={
/pgf/number format/.cd, fixed, fixed zerofill,
precision=2},
xlabel style={font=\color{white!15!black}},
xlabel={$h[-]$},
ymode=log,
yminorticks=false,
ylabel style={font=\color{white!15!black}},
ylabel={$\|\hat{c}(T) - \hat{c}_h(T)\|_{\varepsilon}$},
axis background/.style={fill=white},
xmajorgrids,
ymajorgrids,
legend style={legend cell align=left, align=left, draw=white!15!black, at={(0.06,0.96)}, anchor=north west}
]
\addplot [color=mycolor1, line width=2.0pt, mark=*]
  table[row sep=crcr]{%
0.330150595398318	0.00301683753397609\\
0.183622722177214	0.00123849822247272\\
0.107072661625942	0.000533596135058147\\
0.0594220932193747	0.000246064818815475\\
};
\addlegendentry{$p=1$}

\addplot [color=mycolor2, line width=2.0pt, mark=*]
  table[row sep=crcr]{%
0.330150595398318	0.00135353936629182\\
0.183622722177214	0.00042267851585103\\
0.107072661625942	0.00013965599939964\\
0.0594220932193747	4.28716218389669e-05\\
};
\addlegendentry{$p=2$}

\addplot [color=mycolor3, line width=2.0pt, mark=*]
  table[row sep=crcr]{%
0.330150595398318	0.000177577547630903\\
0.183622722177214	2.32734462398176e-05\\
0.107072661625942	4.30121075696063e-06\\
0.0594220932193747	6.74771395018213e-07\\
};
\addlegendentry{$p=3$}

\addplot [color=mycolor4, line width=2.0pt, mark=*]
  table[row sep=crcr]{%
0.330150595398318	1.59097215769509e-05\\
0.183622722177214	1.31431916843287e-06\\
0.107072661625942	1.38960007279642e-07\\
0.0594220932193747	1.19051788956503e-08\\
};
\addlegendentry{$p=4$}

\addplot [color=mycolor5, line width=2.0pt, mark=*]
  table[row sep=crcr]{%
0.330150595398318	8.86182677637917e-07\\
0.183622722177214	3.26980519070893e-08\\
0.107072661625942	2.01273119775131e-09\\
0.0594220932193747	9.48969625000142e-11\\
};
\addlegendentry{$p=5$}

\addplot [color=mycolor6, line width=2.0pt, mark=*]
  table[row sep=crcr]{%
0.330150595398318	4.67173692996029e-08\\
0.183622722177214	1.15354318632595e-09\\
0.107072661625942	3.85215242073004e-11\\
0.0594220932193747	1.1297130180482e-11\\
};
\addlegendentry{$p=6$}

\draw[black, thick]
(axis cs:0.22, 9.95e-4) -- 
(axis cs:0.25, 9.95e-4) -- 
(axis cs:0.25, 13.875e-4) -- cycle; 
\node[right, xshift=2pt] at (axis cs:0.25, 1.8e-3) {1};

\draw[black, thick]
(axis cs:0.22, 3.5e-4) -- 
(axis cs:0.25, 3.5e-4) -- 
(axis cs:0.25, 5.25e-4) -- cycle; 
\node[right, xshift=2pt] at (axis cs:0.25, 3.5e-4) {2};

\draw[black, thick]
(axis cs:0.22, 3.1e-5) -- 
(axis cs:0.25, 3.1e-5) -- 
(axis cs:0.25, 5.206e-5) -- cycle; 
\node[right, xshift=2pt] at (axis cs:0.25, 4.5e-5) {3};

\draw[black, thick]
(axis cs:0.22, 2.0e-6) -- 
(axis cs:0.25, 2.0e-6) -- 
(axis cs:0.25, 4.001e-6) -- cycle; 
\node[right, xshift=2pt] at (axis cs:0.25, 3.0e-6) {4};

\draw[black, thick]
(axis cs:0.22, 5.5e-8) -- 
(axis cs:0.25, 5.5e-8) -- 
(axis cs:0.25, 12.641e-8) -- cycle; 
\node[right, xshift=2pt] at (axis cs:0.25, 9e-8) {5};

\draw[black, thick]
(axis cs:0.22, 2.0e-9) -- 
(axis cs:0.25, 2.0e-9) -- 
(axis cs:0.25, 5.9441e-9) -- cycle; 
\node[right, xshift=2pt] at (axis cs:0.25, 3.5e-9) {6};

\end{axis}
\end{tikzpicture}%

%% file: Images/p_convergence_2D.tex
%
%
\definecolor{mycolor1}{rgb}{0.00000,0.44700,0.74100}%
\definecolor{mycolor2}{rgb}{0.85000,0.32500,0.09800}%
\definecolor{mycolor3}{rgb}{0.92900,0.69400,0.12500}%
\begin{tikzpicture}[baseline=(current axis.south)]

\begin{axis}[%
width=4.0in,
height=2.5in,
scale only axis,
xmin=1,
xmax=8,
xtick={1, 2, 3, 4, 5, 6, 7, 8},
xlabel style={font=\color{white!15!black}},
xlabel={$p$},
ymode=log,
yminorticks=false,
ylabel style={font=\color{white!15!black}},
ylabel={Error at time $T$},
axis background/.style={fill=white},
xmajorgrids,
ymajorgrids,,
legend style={legend cell align=left, align=left, draw=white!15!black, at={(0.02,0.32)}, anchor=north west}
]
\addplot [color=mycolor1, line width=2.0pt, mark=*]
  table[row sep=crcr]{%
1	0.00169350637826497\\
2	0.000425068070102949\\
3	4.57477436973924e-05\\
4	2.89571382505904e-06\\
5	1.03419375868534e-07\\
6	3.8874174252485e-09\\
7	1.25450800543196e-10\\
8	1.1243932125128e-11\\
};
\addlegendentry{$\|\hat{c}(T)-\hat{c}_h(T)\|$}

\addplot [color=mycolor2, line width=2.0pt, mark=*]
  table[row sep=crcr]{%
1	0.0915074930734084\\
2	0.0446280440201536\\
3	0.00519769034374862\\
4	0.000508829825275142\\
5	2.54810182222433e-05\\
6	1.44617802415783e-06\\
7	5.79494616190884e-08\\
8	2.19454831912453e-09\\
};
\addlegendentry{$\|\hat{c}(T)-\hat{c}_h(T)\|_\mathrm{DG}$}

\addplot [color=mycolor3, line width=2.0pt, mark=*]
  table[row sep=crcr]{%
1	0.00291478870130802\\
2	0.00140658783425552\\
3	0.000162658630481807\\
4	1.60303002493453e-05\\
5	8.06576960518271e-07\\
6	4.5977821902064e-08\\
7	1.84667571715746e-09\\
8	7.12316358140446e-11\\
};
\addlegendentry{$\|\hat{c}(T)-\hat{c}_h(T)\|_{\varepsilon}$}

\end{axis}

\end{tikzpicture}%

%% file: Images/p_convergence_3D.tex
%
%
\definecolor{mycolor1}{rgb}{0.00000,0.44700,0.74100}%
\definecolor{mycolor2}{rgb}{0.85000,0.32500,0.09800}%
\begin{tikzpicture}

\begin{axis}[%
width=4.0in,
height=2.5in,
scale only axis,
xmin=1,
xmax=6,
xtick={1, 2, 3, 4, 5, 6},
xlabel style={font=\color{white!15!black}},
xlabel={$p$},
ymode=log,
ymin=1e-08,
ymax=1.33459,
yminorticks=false,
ylabel style={font=\color{white!15!black}},
ylabel={Errors at time $T$},
axis background/.style={fill=white},
xmajorgrids,
ymajorgrids,
legend style={legend cell align=left, align=left, draw=white!15!black, at={(0.02,0.23)}, anchor=north west}
]
\addplot [color=mycolor1, line width=2.0pt, mark=*]
  table[row sep=crcr]{%
1	0.0312593\\
2	0.00528097\\
3	0.000639199\\
4	6.26365e-05\\
5	5.18058e-06\\
6	4.03206e-07\\
};
\addlegendentry{$\|\hat{c}(T)-\hat{c}_h(T)\|$}

\addplot [color=mycolor2, line width=2.0pt, mark=*]
  table[row sep=crcr]{%
1	1.33459\\
2	0.455875\\
3	0.089213\\
4	0.0130175\\
5	0.00141795\\
6	0.000132875\\
};
\addlegendentry{$\|\hat{c}(T)-\hat{c}_h(T)\|_\mathrm{DG}$}

\end{axis}
\end{tikzpicture}%

%% file: Images/h_convergence_3D_L2.tex
%
%
\definecolor{mycolor1}{rgb}{0.00000,0.44700,0.74100}%
\definecolor{mycolor2}{rgb}{0.85000,0.32500,0.09800}%
\definecolor{mycolor3}{rgb}{0.92900,0.69400,0.12500}%
\definecolor{mycolor4}{rgb}{0.49400,0.18400,0.55600}%

\begin{tikzpicture}

\begin{axis}[%
width=4.0in,
height=2.5in,
scale only axis,
xmode=log,
xmin=0.174792,
xmax=0.452361,
xminorticks=false,
xlabel style={font=\color{white!15!black}},
xlabel={$h[-]$},
xticklabel={\pgfmathparse{exp(\tick)}\pgfmathprintnumber{\pgfmathresult}},
x tick label style={
/pgf/number format/.cd, fixed, fixed zerofill,
precision=2},
ymode=log,
ymin=1e-07,
ymax=0.1,
yminorticks=false,
ylabel style={font=\color{white!15!black}},
ylabel={$\|\hat{c}(T) - \hat{c}_h(T)\|$},
axis background/.style={fill=white},
xmajorgrids,
ymajorgrids,
legend style={legend cell align=left, align=left, draw=white!15!black, at={(0.98,0.02)}, anchor=south east}
]

\addplot [color=mycolor1, line width=1.5pt, mark=*]
  table[row sep=crcr]{%
0.452361	0.0312593\\
0.354585	0.0196621\\
0.279315	0.0121113\\
0.226138	0.00763603\\
0.174792	0.00476778\\
};
\addlegendentry{$p=1$}

\addplot [color=mycolor2, line width=1.5pt, mark=*]
  table[row sep=crcr]{%
0.452361	0.00528097\\
0.354585	0.00249121\\
0.279315	0.00120292\\
0.226138	0.000588376\\
0.174792	0.000286297\\
};
\addlegendentry{$p=2$}

\addplot [color=mycolor3, line width=1.5pt, mark=*]
  table[row sep=crcr]{%
0.452361	0.000639199\\
0.354585	0.000231137\\
0.279315	8.79917e-05\\
0.226138	3.35421e-05\\
0.174792	1.28347e-05\\
};
\addlegendentry{$p=3$}

\addplot [color=mycolor4, line width=1.5pt, mark=*]
  table[row sep=crcr]{%
0.452361	6.26365e-05\\
0.354585	1.71494e-05\\
0.279315	5.12613e-06\\
0.226138	1.52296e-06\\
0.174792	4.65158e-07\\
};
\addlegendentry{$p=4$}

\draw[black, thick]
(axis cs:0.22, 5.5e-3) -- 
(axis cs:0.25, 5.5e-3) -- 
(axis cs:0.25, 7.875e-3) -- cycle; 
\node[right, xshift=2pt] at (axis cs:0.25, 6.18e-3) {2};

\draw[black, thick]
(axis cs:0.22, 4.0e-4) -- 
(axis cs:0.25, 4.0e-4) -- 
(axis cs:0.25, 6.25e-4) -- cycle; 
\node[right, xshift=2pt] at (axis cs:0.25, 5.0e-4) {3};

\draw[black, thick]
(axis cs:0.22, 2.0e-5) -- 
(axis cs:0.25, 2.0e-5) -- 
(axis cs:0.25, 3.906e-5) -- cycle; 
\node[right, xshift=2pt] at (axis cs:0.25, 2.8e-5) {4};

\draw[black, thick]
(axis cs:0.22, 1.0e-6) -- 
(axis cs:0.25, 1.0e-6) -- 
(axis cs:0.25, 2.441e-6) -- cycle; 
\node[right, xshift=2pt] at (axis cs:0.25, 1.56e-6) {5};

\end{axis}
\end{tikzpicture}

%% file: Images/h_convergence_3D_DG.tex
%
%
\definecolor{mycolor1}{rgb}{0.00000,0.44700,0.74100}%
\definecolor{mycolor2}{rgb}{0.85000,0.32500,0.09800}%
\definecolor{mycolor3}{rgb}{0.92900,0.69400,0.12500}%
\definecolor{mycolor4}{rgb}{0.49400,0.18400,0.55600}%
\begin{tikzpicture}

\begin{axis}[%
width=4.0in,
height=2.5in,
scale only axis,
xmode=log,
xmin=0.174792,
xmax=0.452361,
xminorticks=false,
xlabel style={font=\color{white!15!black}},
xlabel={$h[-]$},
xticklabel={\pgfmathparse{exp(\tick)}\pgfmathprintnumber{\pgfmathresult}},
x tick label style={
/pgf/number format/.cd, fixed, fixed zerofill,
precision=2},
ymode=log,
ymin=0.0001,
yminorticks=false,
ylabel style={font=\color{white!15!black}},
ylabel={$\|\hat{c}(T) - \hat{c}_h(T)\|_\mathrm{DG}$},
axis background/.style={fill=white},
xmajorgrids,
ymajorgrids,
legend style={legend cell align=left, align=left, draw=white!15!black, at={(0.98,0.02)}, anchor=south east}
]
\addplot [color=mycolor1, line width=2.0pt, mark=*]
  table[row sep=crcr]{%
0.452361	1.33459\\
0.354585	1.07446\\
0.279315	0.838512\\
0.226138	0.668122\\
0.174792	0.524799\\
};
\addlegendentry{$p=1$}

\addplot [color=mycolor2, line width=2.0pt, mark=*]
  table[row sep=crcr]{%
0.452361	0.455875\\
0.354585	0.272993\\
0.279315	0.165097\\
0.226138	0.100562\\
0.174792	0.0610858\\
};
\addlegendentry{$p=2$}

\addplot [color=mycolor3, line width=2.0pt, mark=*]
  table[row sep=crcr]{%
0.452361	0.089213\\
0.354585	0.0415502\\
0.279315	0.0197975\\
0.226138	0.00960428\\
0.174792	0.00461541\\
};
\addlegendentry{$p=3$}

\addplot [color=mycolor4, line width=2.0pt, mark=*]
  table[row sep=crcr]{%
0.452361	0.0130175\\
0.354585	0.00458739\\
0.279315	0.00169237\\
0.226138	0.000623388\\
0.174792	0.000223727\\
};
\addlegendentry{$p=4$}

\draw[black, thick]
(axis cs:0.22, 5.5e-1) -- 
(axis cs:0.25, 5.5e-1) -- 
(axis cs:0.25, 6.5e-1) -- cycle; 
\node[right, xshift=2pt] at (axis cs:0.25, 5.5e-1) {1};

\draw[black, thick]
(axis cs:0.22, 7.0e-2) -- 
(axis cs:0.25, 7.0e-2) -- 
(axis cs:0.25, 10.25e-2) -- cycle; 
\node[right, xshift=2pt] at (axis cs:0.25, 8.0e-2) {2};

\draw[black, thick]
(axis cs:0.22, 7.0e-3) -- 
(axis cs:0.25, 7.0e-3) -- 
(axis cs:0.25, 11.906e-3) -- cycle; 
\node[right, xshift=2pt] at (axis cs:0.25, 8.0e-3) {3};

\draw[black, thick]
(axis cs:0.22, 4.0e-4) -- 
(axis cs:0.25, 4.0e-4) -- 
(axis cs:0.25, 8.441e-4) -- cycle; 
\node[right, xshift=2pt] at (axis cs:0.25, 5.0e-4) {4};

\end{axis}
\end{tikzpicture}%